\documentclass[10pt]{amsart}
\usepackage{amsmath}
\usepackage{bm}
\usepackage{dsfont}
\usepackage{mathrsfs}
\usepackage{color}
\usepackage[T1, OT1]{fontenc}
\usepackage{hyperref}
\usepackage{subfig}
\usepackage{float}

\usepackage{setspace}
\allowdisplaybreaks
\DeclareSymbolFont{bbold}{U}{bbold}{m}{n}
\DeclareSymbolFontAlphabet{\mathbbold}{bbold}

\numberwithin{equation}{section}

\theoremstyle{plain} \newtheorem{theorem}{Theorem}[section]
\theoremstyle{plain} \newtheorem{lemma}[theorem]{Lemma}
\theoremstyle{plain} 
\theoremstyle{plain} \newtheorem{corollary}[theorem]{Corollary}
\theoremstyle{plain} \newtheorem{proposition}[theorem]{Proposition}
\theoremstyle{remark} \newtheorem{remark}[theorem]{Remark}
\theoremstyle{definition} 
\theoremstyle{definition} 
\theoremstyle{remark} 

\newcommand{ \cU} { \mathcal{U} }

\newcommand{\cI}{\mathcal{I}}
\newcommand{ \cO} { \mathcal{O} }

\newcommand{ \R}{ \mathbb R }
\newcommand{\E}{ \mathbb E}
\newcommand{\N}{ \mathbb N}
\newcommand{\cR}{ \mathcal R}
\renewcommand{\Pr}{ \mathbb P}

\newcommand{\cD}{\mathcal{D}}

\renewcommand{\P}{\mathbb{P}}

\newcommand{\floor}[1]{\ensuremath{\operatorname}{\left \lfloor #1 \right \rfloor }}

\newcommand {\cG}{\mathcal{G}}
\newcommand {\cC}{\mathcal{C}}

\newcommand{\cX}{\mathcal{X}}

\newcommand {\cT}{\mathcal{T}}

\newcommand{\by}{\mathbf{y}}
\newcommand{\bx}{\mathbf{x}}
\newcommand{\bv}{\mathbf{v}}
\newcommand{\bu}{\mathbf{u}}
\newcommand{\bU}{\mathbf{U}}
\newcommand{\bS}{\mathbf{S}}
\newcommand{\bV}{\mathbf{V}}
\newcommand{\bX}{\mathbf{X}}
\newcommand{\bZ}{\mathbf{Z}}
\newcommand{\bz}{\mathbf{z}}

\begin{document}
\title[Free path in an external field]{The free path in a high velocity random flight process associated to a Lorentz gas in an external field}

\author[A. Hening]{Alexandru Hening }
\thanks{A.H. was supported by EPSRC grant EP/K034316/1}
\address{Department of Statistics \\
 1 South Parks Road \\
 Oxford OX1 3TG \\
 United Kingdom}
 \email{hening@stats.ox.ac.uk}
\author[D. Rizzolo]{Douglas Rizzolo }
\thanks{D.R. was supported by NSF grant DMS-1204840}
\address{Department of Mathematics \\
 University of Washington \\
 Box 354350  \\
 Seattle, WA 98195-4350\\
 USA}
 \email{drizzolo@math.washington.edu}
\author[E. Wayman]{Eric S. Wayman}

\address{Department of Mathematics\\
         University  of California\\
         970 Evans Hall \#3840\\
         Berkeley, CA 94720-3840\\
         U.S.A.}

\email{ewayman@math.berkeley.edu}
\maketitle

\renewcommand{\thefootnote}{\fnsymbol{footnote}}
\renewcommand{\thefootnote}{\arabic{footnote}}

\renewcommand{\thefootnote}{\fnsymbol{footnote}}
\renewcommand{\thefootnote}{\arabic{footnote}}

\begin{abstract}
We investigate the asymptotic behavior of the free path of a variable density random flight model in an external field as the initial velocity of the particle goes to infinity.  The random flight models we study arise naturally as the Boltzmann-Grad limit of a random Lorentz gas in the presence of an external field.  By analyzing the time duration of the free path, we obtain exact forms for the asymptotic mean and variance of the free path in terms of the external field and the density of scatterers.  As a consequence, we obtain a diffusion approximation for the joint process of the particle observed at reflection times and the amount of time spent in free flight.
\end{abstract}
\tableofcontents

\section{Introduction}
We study the behavior of a random flight process that arises as the Boltzmann-Grad limit of a random scatterer Lorentz gas in an external field.  The Lorentz gas model was introduced in 1905 by H.A. Lorentz \cite{Lor05} as a model for the motion of electrons in metallic bodies.  Since its introduction, the model has been widely studied by both mathematicians and physicists (see \cite{Det14} for a recent survey).  In this model, a point particle travels in an array of fixed convex scatterers.  When the particle comes in contact with a scatterer it reflects specularly.  This variant is referred to as  the ``hard core'' version.   There is also a  ``soft core'' version in which the scatterers deflect the moving particle via a potential instead of a hard boundary.  Of the many variations on the hard core model, our study is motivated by the version in which an array of spherical scatterers is chosen randomly and the flight of the moving particle is determined by the action of an external field.  The random scatterer Lorentz gas is difficult to study directly even in the absence of an external field.  One approach that has proven tractable, however, is the Boltzmann-Grad limit.  The Boltzmann-Grad limit is a low density limit in which the number of scatterers in a fixed box goes to infinity while, at the same time, the size of each scatterer goes to zero in such a way that the total volume of the scatterers in the box goes to zero.  If the centers of scatterers are placed according to a Poisson process and the rates are chosen appropriately, the asymptotic behavior of the moving particle is described by a Markovian random flight process \cite{G78, S78, S88}.  The Markovian nature of the Boltzmann-Grad limit is due to the following two observations:  re-collisions with scatterers become unlikely as the size of each scatterer goes to zero and the Poisson nature of the scatterer locations which means that knowing the location of one scatterer does not give information about the locations of the other scatterers.  Since analyzing the random Lorentz gas directly is beyond the capability of current techniques, this random flight model is commonly studied in both the mathematics literature \cite{RT99, Vy06, BFS00, BCLM02, BR14} and the physics literature \cite{ADLP10, dW04, vB05, Mu04} to gain insight into the behavior of random Lorentz gas models.

We are interested in the regime in which the particle's velocity is (typically) large.  There are several natural examples of this setting.  The first is an external field which accelerates the particle towards infinity.  For example, the influence of a constant gravitational field has been studied for the random flight process in both a constant density of scatterers \cite{PW79, RT99}, a variable density of scatterers \cite{BR14}, and for the periodic Lorentz gas in two dimensions \cite{CD09}.  The random flight model in a gravitational field has also been used as a model for a particle percolating through a porous medium \cite{WE82}.  A second example of  a high velocity limit is particle motion in a centered, mean zero, isotropic force field \cite{DK09, KR06}.  The soft core Poisson Lorentz gas is an example of this situation.  A discussion of how much of the past work on particles moving in random force fields fits into the high velocity framework can be found in \cite{DK09}.

A unique contribution of this work, is that we study fields that are in general not mean zero, unlike prior work on random force fields.  For example, the constant field directed towards the origin fits into our setting.  In such fields, conservation of energy implies that particles have bounded trajectories.  We are interested in the behavior of the system as the velocity of the particle, or equivalently the energy of the system, is sent to infinity.  Our primary focus is to study the free path of the particle, which is the path of the particle between two reflection times.  The study of the free path has played a key role in much of the work on the hard core Lorentz gas.  To highlight a few examples, the free path played an important role in the recent derivation of the Boltzmann-Grad limit of the periodic Lorentz gas \cite{MS11a, MS11b}, which is substantially more complicated than the corresponding result for the random Lorentz gas. Furthermore, it has also played a central role in the derivation of a diffusive limit for the two dimensional periodic Lorentz gas in a gravitational field \cite{CD09}.  For these reasons, the free path has become an object of interest in its own right, see e.g \cite{MS10, MS14a, MS14b, W12, BZ07, BGW98, GW00, CG03}. We determine precisely the asymptotic behavior of the free path of a particle whose domain of motion is predominantly in the high velocity regime. The primary complication that arises in our study is that the domain of the particle is determined by the external field and, as a result, the particle's velocity is not bounded away from zero.  If the particle enters a part of the domain where it is moving slowly, then  there can be a long time before a reflection occurs, and we must control this in order to analyze the asymptotic behavior of the mean free path.  As a consequence of our analysis of the free path, we obtain several diffusion approximations that illustrate different aspects of the particle's behavior.

We remark that our techniques may be of interest beyond the study of the Boltzmann-Grad limit of a Lorentz gas.  The random flight process we study falls into a general class of models known as transport process.  Models similar to the ones we study here are used in many areas of physics and chemistry as microscopic models for particle behavior.  For example,
they arise in neuron transport \cite{AB99}, the study of semiconductors \cite{BD96}, electron transfer dynamics \cite{HW13,WW05}, and numerous other fields   \cite{BR00,D98,DGL80,DLLS13,Hetal09}.

\subsection{The Model}\label{ss_the_model}
We are primarily interested in the process in dimension three, which is the most interesting spatial dimension and also a dimension where certain technical simplifications occur.  Let us start with a more detailed description of the Boltzmann-Grad limit.  Let $\cU:\R^3 \to \R$ and $g:\R^3 \to (0,\infty)$ be differentiable.  The function $\cU$ will serve as the potential for a conservative force and $g$ will be the density of scatterers.  Fix an energy level $E$ and consider a particle moving in the potential $\cU$ with total energy $E$. By conservation of energy
\[ \frac{m\|\bv(t)\|^2}{2} + \cU(\by(t)) = E,\]
where $\by$ is the particle's position and $\bv$ is its velocity.  In this setting, the Boltzmann-Grad limit can be obtained as follows.  Assume spherical scatterers with radius $1/R$ are placed so their centers are the points of a Poisson process with intensity $R^{2}g$.  Further assume that for almost every initial condition the trajectory of a particle moving in the potential $\cU$ with total energy $E$ is open and has infinite length.  The arguments of \cite{S78, S88} can easily be adapted to include the external field and produce the following result: if the initial position and velocity of the particle is absolutely continuous with respect to Lebesgue measure on the constant energy surface then the distribution of the position and velocity process of the particle converges as $R\to\infty$, in the sense of convergence of finite dimensional distributions, to a Markovian random flight process $(\mathbf{X}(t), \mathbf{V}(t))_{t\geq 0}$ with generator
\begin{multline}\label{eq generator 1} Df(\mathbf{x},\mathbf{v}) =  \mathbf{v} \cdot \nabla_{\mathbf{x}} f(\mathbf{x}, \mathbf{v}) - \frac{1}{m}\nabla_{\mathbf{x}}\cU\cdot \nabla_{\mathbf{v}} f(\mathbf{x}, \mathbf{v}) \\+ g(\mathbf{x})\|\mathbf{v}\| \int_{\mathbf{S}^{2}} (f(\mathbf{x}, \|\mathbf{v}\|\mathbf{u}) - f(\mathbf{x}, \mathbf{v})) \sigma(d\mathbf{u}),  \end{multline}
where $\sigma$ is the normalized surface measure on the unit sphere $\mathbf{S}^2:=\{\bx: \|\bx\|=1\}$, see \cite{S88}.  A similar result holds in dimensions other than $3$, but instead of the integral being against the normalized surface measure it is against a kernel that depends on $\bv$.

\begin{remark}
The assumption that for almost every initial condition the trajectory of a particle with mass $m$ moving in the potential $\cU$ with total energy $E$ is open and has infinite length is essential to the derivation of the Boltzmann-Grad limit.  The intuition for its necessity is that it is possible for a periodic orbit to avoid all of the scatterers in a random arrangement.  If this happens, the effect on the limiting process is that one must keep track of whether or not an orbit has closed, which creates long range dependencies in the limiting process.  The effect of this phenomenon on the Boltzmann-Grad limit has been the subject of study in both the mathematics and physics literature \cite{BHHP96, BHH97, BHH01, DR04, KS98}.  In these cases one may still obtain a limiting process with generator of the form \eqref{eq generator 1}, but additional randomness must be added to array of scatterers.  For example, they can be allowed to move (but not transfer energy), or they can be arranged according to a space-time Poisson process where a scatterer comes into existence, exists for a certain amount of time, and then disappears.
\end{remark}

 The process $(\mathbf{X}(t), \mathbf{V}(t))_{t\geq 0}$ can be constructed in the following way, which explains the name ``random flight process''.  Let $(\by(\bx_0,\bv_0,t))_{t\geq 0}$ with $\bx_0,\bv_0\in \R^3$ and $m\|\bv_0\|^2/2 + \cU(\bx_0)=E$ denote the solution to the initial value problem
 \begin{equation}\label{eq e evo} \left\{\begin{array}{lcc} m\by''  &= & -\nabla \cU(\by) \\ \by(0)& = & \bx_0 \\ \by'(0)& =&\ \bv_0.\end{array}\right.\end{equation}
We construct our process $((\mathbf{X}(t),\mathbf{V}(t)), t\geq 0 )$ recursively as follows.  Set $(\mathbf{X}(0),\mathbf{V}(0)) = (\bx_0,\bv_0)$ and let $T_0=0$. For $k\geq 1$, assuming we have defined $((\mathbf{X}(t),\mathbf{V}(t)))_{0\leq t\leq T_{k-1}}$, we let $\bU_{k-1}$ be independent of this part of the path and uniformly distributed on $\mathbf{S}^2$ and let $T_k$ satisfy
\begin{multline}\label{e_reflection} \P\left(T_k-T_{k-1}>t \mid \bU_{k-1}, ((\mathbf{X}_t,\mathbf{V}_t))_{0\leq t\leq T_{k-1}}\right)\\ =  \exp \left(- \int^{t}_0
g\left(\by(\mathbf{X}(T_{k-1}),\|V(T_{k-1})\|\bU_{k-1},s)\right) \left\|\by'(\mathbf{X}(T_{k-1}),\|\mathbf{V}(T_{k-1})\|\bU_{k-1},s)\right\|ds\right).\end{multline}
For $t\in [T_{k-1},T_{k}]$ we then define
\begin{equation}\label{e_full_trajectory_definition}
\begin{split}
& \bX(t) := \by( \bX(T_{k-1}), \|  \bV(T_{k-1})\|  \bU_{k-1}, t- T_{k-1}) \\
& \bV(t) :=\by'( \bX(T_{k-1}),\|  \bV(T_{k-1})\|  \bU_{k-1}, t- T_{k-1})
\end{split}
\end{equation}
We remark that under very mild assumptions $T_k\to \infty$ a.s. and thus this defines the path of the particle for all times.  Intuitively, $T_k$ defines the $k$th reflection of our particle by a scatterer. By studying the free path  of the particle, we mean to study, in particular, the conditional law of $(\bX(t))_{T_1\leq t \leq T_2}$ given $\bX(T_1) = \bx$ under an appropriate scaling of the parameters in the model. Note that the path of the particle on the time interval $[T_1,T_2]$ is the path between two consecutive reflections. This is why we call it the \textit{free path}.

At this point we make a simple but important observation.  By conservation of energy,
\[ \|  \bV(t)\|   = \sqrt{ \frac{2}{m} (E - \cU(\bX(t))},\]
so that, if we define
\begin{equation}\label{eq E speed} v(\bx) =   \sqrt{ \frac{2}{m} (E - \cU(\bx))},\end{equation}
then
\begin{equation}\label{eq path def} \bX(t) = \by\left( \bX(T_{k-1}), v(\bX(T_{k-1}))\bU_{k-1}, t- T_{k-1}\right).\end{equation}
Since $\bU_{k-1}$ is independent of  $((\mathbf{X}(t),\mathbf{V}(t)))_{0\leq t\leq T_{k-1}}$, this implies that if we define $\bX_k = \bX(T_k)$, then $(\bX_k)_{k\geq 1}$ is a Markov chain.  That the index in this chain starts at $1$ is an artifact of our deterministic choice of $\bV(0)$.  If instead of choosing $\bV(0)=\bv_0$ in the construction above we took $\bV(0) = v(\bX(0))\bU$,
with $\bU$ uniformly distributed on $\mathbf{S}^2$, then $(\bX_k)_{k\geq 0}$ is a Markov chain and its transition operator is
\begin{equation} \label{eq transition} Pf(\bx)  = \E\left[f\left(\by\left(\bx,v(\bx)\bU, N(\bx,\bU)\right)\right) \right],\end{equation}
where $N(\bx,\bu)$ is a random variable with distribution
\begin{equation} \label{eq inter reflection} \P\left(N(\bx,\bu) >t\right) =  \exp \left(- \int^{t}_0 g[\by(\bx,v(\bx)\bu,s)] v[\by(\bx,v(\bx)\bu,s)]ds\right),\end{equation}
and conditional on $\bU=\bu$, $N(\bx,\bU)$ is distributed like $N(\bx,\bu)$.

In order to avoid the notational difficulties and technical assumptions on the geometry of the set $\{\bx\in \R^3: \cU(\bx)=E\}$ that come from working with a general potential $\cU$, we restrict ourselves to the spherically symmetric case where $\cU(\bx) = \cU(r)$ and $g(\bx) = g(r)$, where $r=\|  \bx\|  $.  In this case angular momentum is conserved, which implies that the motion of the particle between collisions with scatterers takes place in a plane.  Consequently, we may restrict our attention to $\R^2$ for some of the analysis of the free path. This lets us work in polar coordinates for these aspects of the analysis.  Throughout let $\R_+ := [0, \infty)$, $\N_0 := \{0,1,2, \dots \} = \N \cup \{0\}$,
and $\cC^k(S)$ be the set of real valued functions on $S$ having $k$ continuous derivatives.  We let $\cU(r)$ be the potential and we assume our particle has
mass $m$ and total energy $E$.  Let $\textbf{r}(t) := \left((r(r_0, \theta, t),\alpha(r_0, \alpha_0, \theta, t)), t \geq 0\right)$ denote the trajectory in polar coordinates of a
particle with mass $m>0$, total energy $E$, initial position $(r_0, \alpha_0)$ and whose velocity vector at time $0$ makes angle $\theta$ with the radial
vector.  Consequently $\textbf{r}(t)$ is the solution to the equations of motion in polar coordinates for a particle of mass $m$ in the potential $\cU$:

\begin{equation}\label{e_motion_polar}
\frac{1}{r}\frac{d}{dt}(r^2\dot\alpha)e_\alpha + (\ddot r-r \dot \alpha^2)e_r = -\frac{\partial_r \cU(r)}{m} e_r
\end{equation}
with initial conditions
\begin{equation}\label{e_initial}
\begin{split}
\textbf{r}(0) &= (r_0, \alpha_0),\\
\dot r(0) &= v(r_0)\cos\theta,\\
\dot\alpha(0) &= \frac{v(r_0)}{r_0} \sin\theta,
\end{split}
\end{equation}
where $e_r$ is the radial unit vector, $e_\alpha$ is the angular unit vector and $\theta$ is the angle the initial velocity makes with the line segment from the origin to the point $\textbf{r}(0)$.  We remark that in this setup it is generally possible for the trajectory to hit the origin in finite time, at which point the equations of motion are no longer well defined.  In this case we continue the path in such a way that momentum is continuous, if possible and such that angular momentum is conserved.  The extension is most easily described in Euclidean coordinates, since it may involve the angular component reentering from $\pm \infty$ in polar coordinates.  Suppose that $\by(t)$ is a trajectory that hits the origin for the first time at time $T>0$.  We extend the path to the interval $[T,2T]$ by defining $\by(t) = -\by(2T-t)$ for $t\in [T,2T]$.  Since $\|\by(2T)\| = \| \by(0)\| \neq 0$, we can continue the path by the equations of motion until it hits the origin again, at which point we repeat the extension procedure.  A straightforward calculation using the radial symmetry of the potential shows that
\[ \frac{d^2}{dt^2} ( -\by(2T-t)) = - \Delta \cU (-\by(2T-t)),\]
so our extension solves \eqref{eq e evo} away from the origin and it is easy to see that $\by'$ is continuous at $T$ if $\lim_{t\uparrow T} y'(t)$ exists.  In this way, we guarantee that all initial conditions lead to trajectories that exist for all time.

Let $\bar{v}(t) := v(r(t))$ denote the speed of the particle as a function of time.  By conservation of energy, we have
\begin{equation}\label{e_energy}
\frac{mv^2(r)}{2} + \cU(r) = E
\end{equation}
and, as a result,
\begin{equation}\label{e_speed_general}
v(r)= \sqrt{\frac{2}{m}(E-\cU(r))}.
\end{equation}
We can also write $\bar{v}(t)$ as a function of  angular and radial velocity:
\begin{equation}\label{e_speed_polar}
\bar{v}(t) = \sqrt{\dot r^2(t) + r^2(t) \dot \alpha^2(t)}.
\end{equation}

\subsection{Heuristics}\label{subsec heuristic}
In this section, we give heuristic arguments to motivate our scaling.  Consider the following simple situation that encapsulates the phenomena in which we are interested.  In particular, suppose that $m=2$, $\cU(\bx) = \|\bx\|$, so that $\cU$ produces a uniform acceleration towards the origin, and that $g(\bx) \equiv 1$.  Suppose we start the particle at the origin with speed  $v_0 = 1$ moving in direction $\mu$.
By conservation of energy the total energy of the particle will be
\[
E(v_0) = mv_0^2/2 = 1.
\]

Since $\cU$ produces a uniform acceleration towards the origin, the maximum distance the particle can move away from the origin is $1$.  Consequently, there is no way to rescale the trajectory of the particle to obtain a diffusive limit: Any scaling of space will cause the trajectory to degenerate, and scaling time alone will produce jumps.  However, we can obtain a diffusive limit if we let $v_0$ (and, consequently, $E$) go to infinity.  This is the type of high velocity limit considered in  \cite{DK09, KR06}.  In those papers, however, the high velocity was used to prevent the particle from being trapped in a bounded domain, while in our setting the particle is still trapped in a bounded domain, but the size of the domain goes to infinity as $v_0$ does.

Treating $v_0$ as a parameter, the energy becomes a function of $v_0$, specifically $E(v_0) = v_0^2$ is the maximum distance the particle can travel from the origin.  Let $(\bX^{v_0}(t))_{t\geq 0}$ be the trajectory of the random flight process with initial conditions $\bX^{v_0}(0) =\mathbf{0}$ and initial velocity $\dot\bX^{v_0}(0) = v_0 \bU$, where $\bU$ is uniformly distributed on $\bS^2$.  In order to obtain a diffusive limit for $(\bX^{v_0}(t))_{t\geq 0}$, we rescale space so that its maximum distance from the origin remains constant.  In particular, we look at $(v_0^{-2} \bX^{v_0}(t))_{t\geq 0}$, which travels at most distance $1$ from the origin.  This is the process for which we analyze the behavior of the free paths.

Since the velocity of the particle is going to infinity, the time between reflections is going to $0$, so we can expect a diffusive limit if we scale time appropriately.  Since $g\equiv 1$, the distance $(\bX^{v_0}(t))_{t\geq 0}$ travels between reflections has an exponential distribution with rate 1, so the distance $(v_0^{-2} \bX^{v_0}(t))_{t\geq 0}$ travels between reflections is of order $v_0^{-2}$.  Writing down the equations of motion for $(v_0^{-2} \bX^{v_0}(t))_{t\geq 0}$ between reflections, we find that its trajectories are the same as those governed by the potential $\cU^{v_0}(\bx) = v_0^{-2} \cU(\bx)$ moving with total energy $E^{v_0} = v_0^{-2}$, initial velocity $v_0^{-1}$, and the scattering rate as measured along these paths is $g^{v_0}(\bx) = v_0^{2}$.

By Brownian scaling, in order to obtain a diffusive limit, we expect that we should scale time so that there are approximately $v_0^4$ reflections per unit time.    Typically, we expect the particle to be distance of order $v_0^2$ from the origin, and if $\bX^{v_0}(t) = v_0^2\bx$ and $t$ is a reflection time, the time until the next reflection is approximately
\[ \frac{1}{\left\| \dot \bX^{v_0}(t)\right\|} =  \frac{1}{\sqrt{ v_0^2 - v_0^2\|\bx\|}} = \frac{1}{v_0\sqrt{ 1 - \|\bx\|}}.\]
Thus the amount of time for $v_0^4$ reflections to occur is of the order $v_0^3$.  This suggests looking for a diffusive limit of $(v_0^{-2} \bX^{v_0}(v_0^3t))_{t\geq 0}$ as $v_0$ tends to infinity, which is what we undertake.

\subsection{General scaling and assumptions}
In general, we allow $g$ to be a small perturbation of the constant density and $\cU$ to produce a small perturbation of the constant field, small being relative to the speed of the particle.  In particular, we will assume the following conditions.

\textbf{(A1)} The path of the particle $(\bX^n(t))_{t\geq 0}$ evolves as the position of a random flight process where the density of scatters is
\[
g_n := \sqrt{n} g,
\]
the potential energy of the field
\[
\cU_n := \frac{1}{\sqrt{n}}\cU
\]
and the total energy of the particle
\[
E_n := \frac{1}{\sqrt{n}}E,
\]
where $g$, $\cU$, and $E$ are fixed and independent of $n$ and $g(\bx)=g(\|\bx\|)$ and $\cU(\bx)=\cU(\|\bx\|)$.  That is, the density of scatterers and potential are spherically symmetric.  By \eqref{e_speed_general} we can write the speed as a function of $E$ and $\cU$; consequently the speed $v$ is also rescaled as
\[
v_n (r) :=  \frac{1}{n^{1/4}}v(r).
\]
The trajectory of the particle with these rescaled parameters we denote by
\[
(\by_n(\bx, v_n(\bx)\bu,t))_{t\geq 0}
\]
when we are working in with the path in 3 three dimensions and by
\[
\left((r_n(r_0, \theta, t),\alpha_n(r_0, \alpha_0, \theta, t)), t \geq 0\right).
\]
when we are considering the path in polar coordinates with respect to its plane of motion.  When no confusion will arise, we will leave the dependence on initial conditions implicit.

In general, subscripts or superscripts of $n$ refer to distributions relative to these scaled parameters.  For example, corresponding to \eqref{eq inter reflection} is the random variable $N^{(n)}(\bx,\bu)$ with distribution
\begin{equation} \label{eq inter reflection n} \P\left(N^{(n)}(\bx,\bu) >t\right) =  \exp \left(- \int^{t}_0 g_n[\by_n(\bx,v_n(\bx)\bu,s)] v_n[\by_n(\bx,v_n(\bx)\bu,s)]ds\right),\end{equation}
and corresponding to \eqref{eq transition} is the Markov chain $(\bX^n_k)_{k\geq 0}$ with transition operator
\begin{equation} \label{eq transition n} P_nf(\bx)  = \E\left[f\left(\by_n\left(\bx,v_n(\bx)\bU, N^{(n)}(\bx,\bU)\right)\right) \right],\end{equation}
where $\bU$ is uniformly distributed on $\bS^2$ and conditionally given $\bU=\bu$, $N^{(n)}(\bx,\bU)$ is distributed like $N^{(n)}(\bx,\bu)$.  Similarly, $(\bX^n(t))_{t\geq 0}$ will be constructed as in \eqref{eq path def} we the rescaled parameters.

\begin{remark}
To make the connection with are heuristic arguments in Section \ref{subsec heuristic}, if we take $v_0 = n^{1/4}$, $g=1$ and $\cU(\bx)=\|\bx\|$ then $(\bX^n(t))_{t\geq 0} = _d ( v_0^{-2}\bX^{v_0}(t))_{t\geq 0}$.  Thus the scaling we do here accounts for increasing the initial speed and scaling space to keep the particle's trajectory contained.  This is the scaling under which we analyze the free path of the particle.  We will scale time separately when we consider diffusive limits.
\end{remark}

\textbf{(A2)} We assume that for all $n \in \N$ the process $( \|\bX^n(t)\|, t \geq 0)$ evolves in a domain $\cD  \subset \R_+$ where $\cD = \R_+$ or $\cD=[h_-,h_+]$, where $0\leq h_- <h_+ <\infty$.    The domain $\cD$ is chosen so
    that $E - U(r) > 0$ for all $r \in \cD^\circ$.  This is equivalent to
  $v(r) > 0$ for all $r \in \cD^\circ$.

\textit{We require the particle to have positive speed in the interior $\cD^\circ$ so the time between reflections
approaches $0$ as $n$ goes to infinity. This way we obtain a nontrivial diffusion limit.}

\textbf{(A3)}  On the boundary $\partial \cD$ we have the following assumptions.
\begin{itemize}
\item If $h_- > 0$ then
\begin{equation*}
\begin{split}
\cU(h_-)&=E\\
-\partial_r\cU(h_-)&:=-\frac{\partial \cU}{\partial r} (h_-)>0\\
\end{split}
\end{equation*}
\item If $h_+ < \infty$ then
\begin{equation*}
\begin{split}
\cU(h_+)&=E\\
-\partial_r\cU (h_+)&<0.\\
\end{split}
\end{equation*}
\end{itemize}
\textit{The conditions on $\cU$ force the speed of the particle to be zero at the respective endpoints $h_-$ and $h_+$, while the
conditions on $\frac{\partial \cU}{\partial r}$ ensure the force field at the endpoints is pointing towards the interior $\cD^\circ$.
Having zero speed at the boundaries and having the force field point `inwards' prevents the particle from leaving the domain $\cD$.
}
\begin{itemize}

\item If $h_- = 0$ then
\[
E-\cU(0) > 0.
\]
\textit{This condition ensures the particle is not trapped at the origin.}
\item If $h_+ = \infty$ we require that for any $\varepsilon >0$, $\inf_{ [h_- + \epsilon, \infty)}( E-\cU(r))> 0.$
\end{itemize}
\textit{This condition implies that for any $\varepsilon >0$, $\inf_{ [h_- + \epsilon, \infty)}v(r)> 0$ which shows that the reflection rate does not go to $0$
as the process goes to infinity. }

\textbf{(A4)}  $\cU\in \cC^1(\cD)$

\textit{This smoothness assumption ensures that the velocity and the acceleration of the particle depend continuously on the position
$r(t)$. If $\cD=[0,h]$ we can relax this condition to $\cU\in \cC^1((0,h])$ so that we can allow potentials of the form $\cU(r)=-\frac{1}{r}$ which are not defined at $0$.}

\textbf{(A5)} The  density $g$ is spherically symmetric and
satisfies
\[
g\in \cC(\cD) \cap \cC^1(\cD^\circ)
\]
together with
\[
\inf_{r\in\cD} g(r)>0.
\]
\textit{We require some smoothness from $g$ because the diffusion limit we get depends on the derivative $g^\prime$. The second assumption is
needed because we do not want to have regions where the the reflection rate goes to $0$.  If one allows $g$ to approach $0$, then the free path may scale differently when the process is started in these regions.  Examples of such situations are dealt with
in Section 8 of \cite{BR14}}.

\subsection{Results}
Our first result considers the asymptotic behavior of the steps of the free path Markov chain $((\bX^n_k,\cT^n_k))_{k\geq 0}$, where $\bX^n_k$ is the location of the particle at the time of the $k$'th reflection and, anticipating our diffusion approximations, $n^{3/4}\cT^n_k$ is the time of the $k$'th collision.  This Markov chain has transition operator
\begin{equation}\label{eq transition time} Q_nf(\bx,t)  = \E\left[f\left(\by\left[\bx,v(\bx)\bU, N^n(\bx,\bU)\right], t+ n^{-3/4}N^n(\bx,\bU)\right) \right].\end{equation}
The following result characterizes the asymptotic mean and covariance structure of the free path chain.
\begin{theorem}\label{theorem chain}
Let
\[
\mu_{n}(\bx,t) := n \E \left[ (\bX^n_1,\cT^n_1) - (\bx,t) \mid (\bX^n_0,\cT^n_0) = (\bx,t) \right]
\]
be the scaled drift of $(\bX^n_k,\cT^n_k)_{k\geq 0}$, and let $S$ be a compact subset of $\cD^\circ$.  Then
\[
\lim_{n \rightarrow \infty} \sup_{(\bx,t): \|\bx\| \in S, t\geq 0} \left|\mu_{n}(\bx,t)  - \left(\frac{-1}{3g(\bx)^2}\left[\frac{ 2\nabla\cU(\bx)}{m v(\bx)^2} + \frac{\nabla g(\bx)}{g(\bx)}   \right] \ , \ \frac{1}{g(\bx)v(\bx)}\right)   \right| =0,
\]
uniformly on $\{\bx : \|\bx\| \in S\}\times [0,\infty)$.  Furthermore, let
\[\begin{split}
\sigma^2_{n,ij}(\bx,t) &:= n \E \left[ \left( X^n_{1,i} - x_i\right) \left( X^n_{1,j} - x_j\right) \mid (\bX^n_0,\cT^n_0) = (\bx,t) \right] \\[10pt]
\sigma^2_{n,it}(\bx,t) &:= n \E \left[ \left( X^n_{1,i} - x_i\right) \left( \cT^n_1- t\right) \mid (\bX^n_0,\cT^n_0) = (\bx,t) \right] \\[10pt]
\sigma^2_{n,t}(\bx,t) &:= n \E \left[ \left( \cT^n_1-t\right)^2  \mid (\bX^n_0,\cT^n_0) = (\bx,t) \right].
\end{split}\]
Then
\begin{enumerate}
\item\label{space}
$\displaystyle
\lim_{n \rightarrow \infty} \sup_{(\bx,t) : \|\bx\|\in S, t\geq 0 } \left| \sigma^2_{n,ij}(\bx,t) - \frac{2}{3g(\bx)^2}\delta_{ij} \right| =0,
$ where $\delta_{ij} = \mathds{1}\{i=j\}$.
\item \label{space/time}
$\displaystyle
\lim_{n \rightarrow \infty} \sup_{(\bx,t) : \|\bx\|\in S, t\geq 0 } \sigma^2_{n,it}(\bx,t)=0.
$
\item \label{time}
$\displaystyle
\lim_{n \rightarrow \infty} \sup_{(\bx,t) : \|\bx\|\in S, t\geq 0 } \sigma^2_{n,t}(\bx,t)=0.
$
\end{enumerate}
\end{theorem}

This theorem is a combination of Lemmas \ref{l_R_drift_convergence} and \ref{l_R_variance_convergence} below.  Heuristically, this says that both $\cU_n$ and $g_n$ impart a drift towards areas where the corresponding function has a smaller value.  That is, the particle prefers to move towards areas where it travels quickly and where there are few scatterers, which may be competing influences.  It is interesting to note that the potential $\cU$ does not appear in the covariance terms.  It can, and will, effect the diffusion coefficient of the limiting diffusion only through an overall time change.

With Theorem \ref{theorem chain} in hand, it is straightforward to obtain the following diffusion approximation, which separates out the effects of the particle's position at times of reflection from the effects of speed at which the particle is moving.  Let $\cD_{3} = \{\bx\in \R^3 : \|\bx\|\in \cD\}$.

\begin{theorem}\label{t_step_process_convergence_a}
Let $(\cX_t)_{t\geq 0}$ be a diffusion on $\cD_{3}$ whose generator $G$ acts on functions $f \in C^2(\cD_{3})$ with compact support in $\cD_{3}^\circ$ by
\begin{equation}\label{e_G_definition}
Gf(\bx) =
\frac{1}{3g^2(\bx)}\Delta f(\bx)  - \frac{1}{3g^2(\bx)}\left(
\frac{\nabla g(\bx)}{g(\bx)}
+ \frac{2\nabla\cU(\bx )}{mv(\bx)^2}\right)\cdot \nabla f(\bx)
\end{equation}
and killed if/when $\cR: = \|\cX\|$ hits the boundary of $\cD$.  Consider any $l,u \in \cD^\circ$ with
$l <  u $ and start the process   $\left((\bX^n_k,\cT^n_k) ,k \in \N_0\right)$ at $(\bx, 0)$, where $l<\|\bx\| < u$.  Define the
stopping times
\[\tau^n_{l,u} := \inf \{k \in \N_0: \|\bX^n_k\|  \notin [l,u] \}\] and
\[\tau_{l,u} := \inf \{ t \geq 0: \|\cX_t\|\notin [l,u]\}.
\]
Then, as $n\rightarrow \infty$, the family of continuous time processes $((\bX^{n}_{\floor{nt} \wedge \tau^n_{l,u}},\cT^{n}_{\floor{nt} \wedge \tau^n_{l,u}}), t
\geq 0)$
converges
in distribution on the Skorokhod space $D(\mathbb{R}_+, \mathbb{R})$ to the diffusion
\[ \left( \left(\cX_{t \wedge \tau_{l,u}}, \int_0^{t \wedge \tau_{l,u}} \frac{ds}{g(\cX_s)v(\cX_s)}\right), t \geq 0 \right)\]
with initial position $(\bx,0)$.  This convergence happens jointly with the convergence of the hitting times, $n^{-1}\tau^n_{l,u} \Rightarrow \tau_{l,u}$.
\end{theorem}

This theorem is proved in Section \ref{s_skeleton_process}, where it is restated for convenience.

\begin{remark}\label{r_nocutoff_skeleton}
The convergence above also holds without stopping near a boundary point of $\cD$ that is inaccessible for the diffusion $\cR=\|\cX\|$.  See Section \ref{s_boundaries} for how one can determine when a point is inaccessible.  In particular, the boundaries are inaccessible for the constant acceleration towards the origin and for Newtonian gravity centered at the origin.
\end{remark}

\begin{remark}
This type of limit theorem, that gives joint convergence of a process observed when it changes direction and the time spent between such changes up to that point, is typical of the continuous time random walk literature, see e.g. \cite{MS14}.  Although our models can naturally be viewed as transport processes, our results on the free path of the particle suggest an approach to diffusive limits that has more in common with that used to study Continuous Time Random Walks
(CTRW) \cite{MS04, MS14} than that typically used to study transport processes.  CTRW models also arise in many physical applications (see e.g.
\cite{Wetal14,MK00,BCDS06,SBMB03}) and have been particularly useful in studying anomalous diffusions. The advantage of this approach is that, on the
diffusive scale, we may easily distinguish between the effects of the particle's displacement between collisions and the effects of the speed at which the
particle is traveling.  In this way, our approach is similar to the one taken in \cite{BR14}, and our work can be seen as an extension of the methods used there.  This approach is also similar to the approach in \cite{CD09} to studying the two dimensional periodic Lorentz gas in a gravitational field, though, of course, our setting is far simpler.
\end{remark}

By inverting the time process $(\cT^n_k)_{k\geq 0}$ and showing that the free path is typically not far from the straight line between its endpoints, we can arrive at a diffusion approximation for $(\bX^n(t))_{t\geq 0}$, the trajectory of the particle.  In particular, we obtain the following result.

 \begin{theorem}\label{t_full_trajectory_convergence_a}
 Let $(\bX^n(t))$ denote the trajectory of the particle and fix $l,u \in
 {\cD}^\circ$ with $l < u$.  Suppose that $\|\bX^n(0)\| \in (l,u)$ and define
 \[
 \iota^n_{l,u} :=\inf \{  t: \|\bX^n(t)\| \notin [l,u]\}.
 \]
We have the following convergence
 in distribution on the Skorokhod space $D(\mathbb{R}_+, \mathbb{R})$:
 \[
 \left( \bX^n((n^{3/4} t)\wedge \iota^n_{l,u}) \right)_{t\geq 0} \rightarrow \left(\cX(\Omega(t)\wedge
 \tau_{l,u})\right)_{t\geq 0}
 \]
 where $\cX$ is as in Theorem \ref{t_step_process_convergence} and $\Omega$ is the time change given by
 \[
 \Omega(t) := \cI \left(\int^{\cdot}_0\frac{ds}{g( \cX(s \wedge  \tau_{l,u}))v( \cX(s \wedge  \tau_{l,u})) } \right)(t)
 \]
and $\cI$ is the inverse operator defined by $\cI(f)(t) = \inf \{ s: f(s) > t\}$.
Furthermore, the time changed process $\left(\cX(\Omega(t)), t \geq 0)\right)$ is a diffusion process whose generator $\mathscr{G}$ acts on functions $f\in C^2(\cD_3)$ with compact support in $\cD_3^\circ$ by
\[ \mathscr{G}f(\bx) = g(\bx)v(\bx)Gf(\bx) = \frac{v(\bx)}{3g(\bx)}\Delta f(\bx)  - \frac{v(\bx)}{3g(\bx)}\left(
\frac{\nabla g(\bx)}{g(\bx)}
+ \frac{2\nabla\cU(\bx )}{mv(\bx)^2}\right)\cdot \nabla f(\bx)\]
 \end{theorem}

This theorem is proved in Section \ref{s_full_trajectory}, where it is restated for convenience.

 \begin{remark}\label{r_nocutoff_natural_scale}
 The stopping at $u$ and/or at $l$ can be removed when the left and/or right boundary points of $\cD$ are inaccessible for the diffusion $\cR$. See Remark
 \ref{r_nocutoff_skeleton} and Section \ref{s_boundaries}.
 \end{remark}

\begin{remark}
The cutoffs are necessary because of the generality the potentials and scattering densities we allow permit very different behaviors at the boundary.  In some
cases we the boundaries to be inaccessible (see Section \ref{s_boundaries}), while in other cases we expect the boundaries to be reflecting as in \cite{RT99}
and in yet others, like $\cU(r)=-r^{-2}$, the origin should trap the particle.  We leave a detailed investigation of the boundary as an open problem.
\end{remark}

\begin{remark} \label{r:C91}
Our results have much in common with \cite{C91}, but there are key differences at both heuristic and technical levels and, as a result, there is no overlap
between the precise results.  For instance, our models do not satisfy the underlying technical assumptions of \cite{C91}.  The function $Q^{-1}$ from
\cite{C91}, which controls the reflection rate, has to be differentiable up to the boundary of the
region where the motion takes place (see assumption (H2) from \cite{C91}). In our case we have singularities at the origin and when
$\cD=[0,h]$ we also have singularities at the boundary of the domain. Furthermore, the speed
of the particle must be bounded away from $0$ and $\infty$ in \cite{C91}, while there is no such restriction in our model.  This is related to the fact that
\cite{C91} restricts the particle's motion by imposing a reflecting boundary while in our case the domain of the particle's motion is completely determined by
$\cU$ and $E$.
\end{remark}

\begin{remark}
The form of the generator \eqref{eq generator 1} puts the random flight process in the class of processes known as transport processes.  There is a substantial literature on diffusion approximations of transport processes and, more generally random evolutions.  Very general approaches and results can be found, for example in \cite{K73, Pap75, BPL79, C91}.  In order to apply these general techniques the generator needs to be smoothed to control for regions where the particle's velocity is low.  Nonetheless, it is possible to adapt these approaches to obtain Theorem \ref{t_full_trajectory_convergence_a}.  However, our results on the free path are stronger than the diffusion approximation results and cannot be recovered by these general methods.  Moreover, these methods do not clarify the connection with the free path chain.
\end{remark}


\section{The free path of the particle}\label{s_preliminaries}
In this section we analyze the jumps of the free flight chain $((\bX^n_k,\cT^n_k))_{k\geq 0}$ whose transition operator is given by \eqref{eq transition time}.  Our proofs are complicated by the fact that for a general potential $\cU$ it
is not possible to explicitly find the trajectory of the particle in the absence of collisions and this makes it difficult to determine how long this path spends in domains where it is traveling slowly.  Our first lemma will give us local control over the how the radial part of the path behaves.
\begin{lemma}\label{l_r_taylor_expansion}
For any $r_0 \in \cD^\circ$ and $\theta \in [-\pi, \pi]$
\[
r(t):= r(r_0, \theta, t) = r_0 + v_0\cos(\theta)  \cdot t + \frac{1}{2}\left(\frac{-\partial_r\cU(r(\tau))}{m} +
\frac{v_0^2r_0^2\sin^2(\theta)}{r(\tau)^3}\right)t^2
\]
for some $0 \leq \tau \leq t$ depending on $r_0$,  $\theta$ and $t$.
\end{lemma}
\begin{proof}

By definition $\textbf{r}(t) := \left(r(r_0, \theta, t),\alpha(r_0, \alpha_0, \theta, t)\right)$ is the solution to the equations of motion in polar
coordinates for a particle of mass $m$ in the potential $\cU$.
Equation \eqref{e_motion_polar} implies
\begin{equation}\label{e_d_angular}
\frac{1}{r}\frac{d}{dt}(r^2\dot\alpha)=0
\end{equation}
and
\begin{equation}\label{e_polar}
\ddot r-r \dot \alpha^2 =  -\frac{\partial_r \cU(r)}{m}.
\end{equation}
As a result of \eqref{e_d_angular} and \eqref{e_initial}
\begin{equation}\label{e_d_ang}
\dot \alpha(t) = \frac{v(r_0)r_0\sin\theta}{r(t)^2}.
\end{equation}
By \eqref{e_polar} and \eqref{e_d_ang}
\begin{equation}\label{e_ddot_r}
\ddot r(t) = \frac{-\partial_r\cU(r(t))}{m} + \frac{v(r_0)^2r_0^2\sin^2(\theta)}{r(t)^3}.
\end{equation}
Taylor expanding $r(t)$ and using \eqref{e_polar}, \eqref{e_initial} together with \eqref{e_d_ang} yields
\begin{equation}\label{e_taylor}
r(t) = r_0 + v(r_0)\cos(\theta)  \cdot t + \frac{1}{2}\left(\frac{-\partial_r\cU(r(\tau))}{m} +
\frac{v(r_0)^2r_0^2\sin^2(\theta)}{r(\tau)^3}\right)t^2.
\end{equation}
\end{proof}

\begin{remark}\label{r_S_Sprime}
Throughout the remainder of this section we let $S, S^\prime$ be closed intervals satisfying
\[
S \subset (S^\prime)^\circ \subset S^\prime \subset \cD^\circ.
\]
 Also, for any $\delta>0$, we define
\begin{equation}\label{e_lambda_delta_definition}
\Lambda^n_\delta(S) = \inf_{(\bx,\bu) : \|\bx\| \in S, \bu\in \bS^2 }\inf_{t \geq 0} \{t: \| \by_n(\bx, \bu, t) - \bx \| \geq \delta \}
\end{equation}
to be the shortest time it takes for the particle's displacement from its initial position to be at least $\delta$ when started inside $S$.

If
\[
0<\delta < d(S,(S^\prime)^c) : = \inf \left\{|x-y|: x\in S, y\in \cD\setminus S^\prime\right\},
\]
then for all $(\bx,\bu,t)\in \{\bx:\|\bx\|\in S\}\times\bS^2\times [0,\Lambda^n_\delta(S)]$ one has
\[
\|\by_n(\bx,\bu,t)\|\in S^\prime.
\]
Since $\| \dot \by_n \| = n^{-1/4}v(\by)$, we have
\[
\Lambda^n_\delta (S) \geq n^{1/4} \cdot \frac{\delta}{\sup_{r_0 \in S^\prime}v(r_0)} >0
\]
where $\sup_{r_0 \in S^\prime}v(r_0)<\infty$ since $S'$ is bounded away from $\partial D$.
\end{remark}

The next lemma shows that the scaled trajectories converge to the starting point, uniformly over a fixed time interval.
\begin{lemma}\label{l_r_n_uniform_convergence}
Fix $T >0$.  Then
\begin{align*}
\lim_{n \rightarrow \infty} \sup_{(\bx,\bu,t) \in \{\bx : \|\bx\|\in S\} \times \bS^2 \times [0,T]} \left\|\by_n(\bx,\bu,t) - \bx \right\|
  = 0\\[10pt]
\lim_{n \rightarrow \infty} \sup_{(\bx,\bu,t) \in \{\bx : \|\bx\|\in S\} \times \bS^2 \times [0,T]} \left\|\sqrt{n}\ddot\by_n(\bx,\bu,t) - \frac{\nabla\cU(\bx)}{m} \right\|
  = 0\\[10pt]
 \lim_{n \rightarrow \infty} \sup_{(r_0,\theta,t) \in S \times [-\pi,\pi]\times [0,T]} \left|r_n(r_0,\theta,t) - r_0 \right|
 = 0
\end{align*}
and
\begin{equation}\label{e_uniform_convergence_of_d^2rn}
 \lim_{n \rightarrow \infty }\sup_{(r_0,\theta,t) \in S \times [-\pi,\pi]\times [0,T]}  \left| \sqrt{n}\ddot
 r_n(r_0, \theta,t) -
 \left(- \frac{\partial_r \cU(r_0)}{m} +  \frac{v^2(r_0) \sin^2(\theta)}{r_0}\right) \right|  = 0.
\end{equation}

\end{lemma}
\begin{proof}
The claims about $r_n$ are slightly more complicated, so we prove those and leave the more obvious claims about $\by_n$ to the reader.  Let $\delta>0$ be such that $\delta\leq d(S,(S^\prime)^c)$. By Remark \ref{r_S_Sprime} there exists $M \in \mathbb{N}$ large enough such that
$\Lambda^n_\delta(S) \geq T$ whenever $n \geq M$. Equivalently, for $n
\geq M$ we have
\[
|r_n(r_0, \theta,t) - r_0| \leq \delta
\]
for all $(r_0,\theta,t ) \in S \times [-\pi,\pi] \times [0,T]$.  This proves the uniform convergence of $r_n(r_0,\theta,t)$ to $r_0$.
Define
\begin{equation}\label{e_psi}
\psi(r_0, r_0, \theta):= - \frac{\partial_r \cU(r_0)}{m} +  \frac{v^2(r_0) r_0^2 \sin^2(\theta)}{r_0^3}
\end{equation}
and note by \eqref{e_ddot_r} that
\[
\ddot r_n (r_0, \theta,t) = \frac{1}{\sqrt{n}}\psi(r_n(t), r_0, \theta).
\]
 $S^\prime$ is bounded away from $0$ and  $\cU \in \cC^1(\cD^\circ)$ imply that $\psi$ is
uniformly continuous on $S^\prime \times S \times [-\pi, \pi]$.  By construction $r_n(t) \in S^\prime$ for all $n \geq M$ and $t \in [0,T]$. Because
$r_n(r_0,\theta, t)$ converges uniformly to $r_0$ on compact sets, we have
\begin{equation}\label{e_uniform_convergence_of_d^2rn}
 \lim_{n \rightarrow \infty }\sup_{(r_0,\theta,t) \in S \times [-\pi,\pi]\times [0,T]} \left\{ \left| \sqrt{n}\ddot
 r_n(r_0, \theta,t) -
 \psi(r_0, r_0, \theta) \right|  \right\} = 0
\end{equation}
where
\[
\psi(r_0,r_0, \theta)=- \frac{\partial_r \cU(r_0)}{m} +  \frac{v^2(r_0) \sin^2(\theta)}{r_0}
\]
by \eqref{e_psi}.  This completes the proof.
\end{proof}
Since the free flight chain tracks the process at reflection times, we want to apply the estimates for \eqref{e_taylor} between reflections.  By
\eqref{e_reflection} and the rescaling, we know that for every $k$, if at $T_k$ the particle is at position $(r_0,\alpha_0)$ and reflects in the direction
given by the angle $\theta$, then $T^n_{k+1} - T^n_k$ is distributed like the random variable $N^{(n)}(r_0, \theta)$ which we define by
\begin{equation}\label{e_N_distribution}
\Pr(N^{(n)}(r_0, \theta)>t) =\exp\left(-\int_0^t  n^{1/4}\lambda(r_n(r_0, \theta, s))\,ds\right)
\end{equation}
where
\begin{equation}\label{e_lambda_definition}
\lambda(r_0) := g(r_0)v(r_0).
\end{equation}
Throughout, we will often suppress the $r_0$ and $\theta$ dependencies of $N^{(n)}(r_0, \theta)$ and write $N^{(n)}$ or $N^{(n)}(\theta)$ when no confusion
will arise.  To go back and forth between Euclidean and polar coordinates, note that we have the identity in distribution
\[ N^{(n)}(\bx,\bu) =_d N^{(n)}\left(\|  \bx\|  ,\cos^{-1}\left( \frac{\bu\cdot \bx}{\|  \bx\|  }\right)\right),\]
where $N^{(n)}(\bx,\bu)$ is defined as in \eqref{eq inter reflection} with the appropriately rescaled parameters.  Consequently, for studying the asymptotic duration of the free path, there is no loss in studying $N^{(n)}(r,\theta)$ in a uniform way over $r$ and $\theta$.

For many of our proofs we require estimates that show the time between reflections approaches $0$ with high probability as the scaling factor $n$ goes to
infinity.  This will allow us to apply the local estimates from the expansion of $r(t)$ in Lemma \ref{l_r_taylor_expansion} and to show the free flight process does not undergo large jumps.  We first prove some bounds on the moments of $N^{(n)}(r_0, \theta)$.
 \begin{lemma}\label{l_bounded_moments}
The family
 \[
 \{ n^{1/4}N^{(n)}(r_0, \theta): (r_0, \theta, n) \in S\times\ [-\pi,\pi] \times \mathbb{N} \}
 \]
 is bounded in $L^p$ for $1 \leq p < \infty$.
 \end{lemma}

 \begin{proof}
We will assume that $\cD = [0,h]$ or $\cD=[0,\infty)$. The cases $\cD=[h_-,h_+]$ and $[h,\infty)$ can be treated similarly.

 \textit{Case I: $\cD =[0,h]$}

By Assumptions (A3) and (A4) there exist $\delta>0$ and $U_\delta>0$ such that
 \begin{equation}\label{e_bound_U}
 \min_{[h-\delta,h]}|\partial_r U (r_0)|\geq U_\delta.
 \end{equation}
 Let $S\subset D^\circ$ be a compact set, let $\eta>0$ and assume our particle enters the annulus $A_{\eta}$ with inner radius $h-\eta$ and
 outer radius $h$ at time $t_{0}$. By
  \eqref{e_speed_general}
  \begin{equation}\label{e_speed_eta}
  v(h-\eta)=\sqrt{\frac{2}{m}(E-\cU(h-\eta))}.
  \end{equation}
  Using \eqref{e_polar} and \eqref{e_d_ang}
  \begin{equation}\label{e_ddot_r_bounds}
  \ddot r(t) = -\frac{\partial_r \cU (r(t))}{m} + r(t)\dot\alpha^2(t)\leq -\frac{\partial_r \cU (r(t))}{m} + \frac{[v(h-\eta)]^2}{r}.
  \end{equation}
  Let
  \[
  t_e:=\inf\{s>t_0~:~r(s)=h-\eta\}
  \]
  be the time when the particle exits the annulus. By Assumption (A3)
  \[
  \cU(h)=E.
  \]
  Since $\cU$ is continuous this means that we can make $v(r)$ as small as we like if we are close enough to $r=h$.
  This together with \eqref{e_bound_U}, \eqref{e_speed_eta},  and \eqref{e_ddot_r_bounds} implies that there exist $\gamma>0, m_\gamma>0$ such
  that
  \[
  \ddot r(t) \leq -m_\gamma
  \]
  whenever $r(t)\in[h-\gamma,h]$. Set $\eta =\gamma$.
  Note that $t_e-t_0$ is bounded above by the time it would take a particle started at $r=h-\gamma$ with speed $\dot r(t_0)>0$ pointed along the radius  and
  with acceleration
  $\ddot r = -m_\gamma<0$ to return to $r=h-\gamma$. Thus,
 \begin{equation}\label{e_t_e_t_0}
  t_e-t_0\leq 2 \frac{\dot r(t_0)}{m_\gamma}<\infty.
 \end{equation}

  Next, define
  \[
  t_r:=\inf\{s>t_e~:~r(s)=h-\gamma\}.
  \]
  This is the first return time to the annulus $A_\gamma$. We want to bound $t_r-t_e$ below.  If $\dot r (t_0)\leq 0$ then the particle
  would not
  spend any time in the annulus $A_\gamma$, that is $t_e-t_0=0$. Therefore, we can assume that $\dot r(t_0)>0$. By conservation of
  angular momentum and conservation of energy we have that $v(t_0)=v(t_e)$ and $\dot \alpha(t_e)=\dot \alpha(t_0)$. This implies that $\dot r(t_e)=-\dot
  r(t_0)<0$. Since $\ddot r$ is finite on $S$ we immediately get that
  \begin{equation}\label{e_t_r_t_e}
  t_r-t_e\geq\frac{2\dot r(t_0)}{\sup_S\left|\ddot r\right|}\geq \frac{2\dot r(t_0)}{\frac{\sup_S|\partial_r \cU (r)|}{m}
  +
  \sup_S|r\dot\alpha^2|}\geq \frac{2\dot r(t_0)}{\frac{\sup_S|\partial_r \cU (r)|}{m}
  +
  \sup_S\frac{v^2(r)}{r}}>0.
  \end{equation}

 Combining \eqref{e_t_e_t_0} and \eqref{e_t_r_t_e},
 \begin{equation}\label{e_time_bound}
 \sup_{S\times[-\pi,\pi]}\frac{t_e-t_0}{t_r-t_e} = \sup_{S\times[-\pi,\pi]} \frac{ 2 \frac{\dot r(t_0)}{m_\gamma}}{\frac{2\dot r(t_0)}{\frac{\sup_S|\partial_r
 \cU (r)|}{m}
  +
  \sup_S\frac{v^2(r)}{r}}}\leq \frac{\frac{\sup_S|\partial_r \cU (r)|}{m}
  +
  \sup_S\frac{v^2(r)}{r}}{m_\gamma}<\infty
 \end{equation}
 where we assume that if $t_0=\infty$ then $t_e-t_0=0$ and
 \[
 \frac{t_e-t_0}{t_r-t_e}=0.
 \]

 Note that the upper bound from \eqref{e_time_bound} is the same if we use the rescaled trajectory of the particle $(r_n(t),\alpha_n(t))$. This follows from
 the way we found $m_\lambda$ and by Assumption (A1).

By Assumption (A2) we know that
 \[
 \inf_S v(r_0) >0.
 \]
 Suppose that $t \geq \sup_S(t_e - t_0) + \inf_S(t_r - t_e)$.\\
 Using \eqref{e_time_bound} together with the fact that the worst case scenario is when the particle spends the longest possible time in the
 `bad region' $A_\gamma$ and the least amount of time in the `good region' $\cD\setminus A_\gamma$, we have for some $c>0$ that
 \begin{equation}\label{e_bound_P(N>t)}
 \begin{split}
 \Pr\{N^{(n)}>t\}&=\exp\left(-\int_0^tn^{1/4} g(r_n)v(r_n)\,ds\right)\\
 &\leq \exp\left(-\int_0^{\left(\frac{1}{1+c}\right)t}n^{1/4} \inf_S g(r) \inf_S v(r)\,ds\right)\\
 &=\exp\left(-\frac{1}{1+c}t n^{1/4} \inf_S g(r) \inf_S v(r) \right)
\end{split}
\end{equation}
which decays exponentially in $n$ as $n\rightarrow\infty$ as long as $t$ is large. Therefore,
\begin{equation}
\begin{split}
\E\left[\left(n^{1/4}N^{(n)}\right)^p\right] &= \int_0^\infty p t^{p-1} \Pr\left(n^{1/4}N^{(n)}>t\right)\,dt \\
&\leq \int_0^\infty p t^{p-1} \exp\left(-\frac{1}{1+c}t \inf_S g(r) \inf_S v(r) \right)\,dt\\
&< \infty.
\end{split}
\end{equation}
Since the bound above does not depend on $r$, $\theta$ or $n$ we get that
\[
\sup_{S\times [-\pi,\pi]\times \N}\E\left[\left(n^{1/4}N^{(n)(r_0,\theta))}\right)^p\right]<\infty.
\]
\textit{Case II: $\cD=[0,\infty)$}

Set  $g_{\min}:=\inf_{r\in\mathcal{D}}g(r)$. We know by Assumption (A3) that there exists $\bar\delta>0$ such that $\inf_{\mathcal{D}} v(r)\geq
\bar\delta$
 so by \eqref{e_N_distribution}
 \[
 \Pr\{N^{(n)}>t\}\leq \exp\left(-n^{1/4}\bar\delta t g_{\min}\right)
 \]
 which, like before, forces
\begin{equation}
\begin{split}
\E\left[\left(n^{1/4}N^{(n)}\right)^p\right] &= \int_0^\infty p t^{p-1} \Pr\left(n^{1/4}N^{(n)}>t\right)\,dt \\
&\leq \int_0^\infty p t^{p-1} \exp\left(-\bar\delta t g_{\min}\right)\,dt \\
&< \infty.
\end{split}
\end{equation}
Similarly to Case I, because the bound above does not depend on $r$, $\theta$ or $n$ we get that
\[
\sup_{S\times [-\pi,\pi]\times \N}\E\left[\left(n^{1/4}N^{(n)(r_0,\theta))}\right)^p\right]<\infty.
\]

\end{proof}

 Lemma \ref{l_bounded_moments} also provides us with the following corollary which will prove useful in showing that the probability that
 $N^{(n)}(r_0, \theta)$ is larger than any fixed value decays rapidly as the scaling parameter $n$ goes to infinity.

 \begin{corollary}\label{c_PN_limit}
For all $k \in \R$ and for all $\varepsilon > 0$,

 \begin{equation*}
 \lim_{n \rightarrow \infty} \sup_{(r_0, \theta) \in S \times [-\pi, \pi]} n^k\Pr \left\{N^{(n)}(r_0, \theta)  \geq \varepsilon \right\}
 =0
 \end{equation*}
 \end{corollary}

 \begin{proof}
 This is immediate from Lemma \ref{l_bounded_moments}.

 \end{proof}

 In addition, we have also the following corollary that shows that tails of the moments of $N^{(n)}(r_0, \theta)$ decay rapidly as well.

 \begin{corollary}\label{c_EN_tails}
 For all $k \in \R,l \geq 1$ and  $\varepsilon > 0$
 \[
 \lim_{n \rightarrow \infty} \sup_{(r_0, \theta) \in S \times  [-\pi, \pi]}n^k \E \left[\left( N^{n}(r_0, \theta) \right)^l \mathds{1}_{\{
 N^{(n)} > \varepsilon \}} \right]  = 0.
 \]
 \end{corollary}
 \begin{proof}
Cauchy-Schwarz together with Lemma \ref{l_bounded_moments} and Corollary \ref{c_PN_limit} yield the desired result.

 \end{proof}

 From Lemma \ref{l_bounded_moments}, we have the following estimates on the moments of $N^{(n)}$.
 \begin{lemma}\label{l_N_moments}

Let $1 \leq p < \infty$.  Then
\begin{equation}
\lim_{n \rightarrow \infty} \sup_{(r_0, \theta) \in S \times [-\pi, \pi]} \left| \E\left(n^{p/4}[N^{(n)}(r_0,\theta)]^p\right) -
\int^{\infty}_0pt^{p-1}\exp\left(-g(r_0)v(r_0)t\right)\,dt\right| = 0.
\end{equation}
\end{lemma}
 \begin{proof}
 First let $M \in \R_+$ and note that the truncated moments $\E \left[ n^{p/4}N^{(n)}(r_0,\theta)^p\right ] \wedge M^p]$ can be written as
 \begin{equation}
 \begin{split}
 \E \left[ n^{p/4}N^{(n)}(r_0,\theta)^p \wedge M^p \right] &= p \int^{\infty}_0 t^{p-1}\P \left[ n^{1/4}N^{(n)}(r_0,\theta) \wedge M > t\right]dt\\
 &=  p \int^{M}_0 t^{p-1}\P \left[ n^{1/4}N^{(n)}(r_0,\theta) > t\right]dt.
 \end{split}
 \end{equation}
 Making the change of variables $u = n^{1/4}s$, we have
 \begin{equation}
 \Pr\left(n^{1/4}N^{(n)}(r_0, \theta)>t\right) = \exp\left(- \int^t_0 g(r(r_0, \theta, u/n^{1/4})) v(r_0, \theta,u/n^{1/4}) \right)du.
 \end{equation}
 Both $r = r(r_0,\theta, t)$ $v=v(r_0,\theta, t)$ are continuous functions on $\mathcal{O}:= S \times [-\pi, \pi]
 \times [0, M]$.
 Since $\mathcal{O}$ is compact, $r$ and $v$ are in fact uniformly continuous on $S$.   By Assumption (A5) $g$ is continuous on $S$, and therefore
 is
 uniformly continuous. This implies that $g \circ r$ is uniformly continuous on $\mathcal{O}$. By Lemma \ref{l_r_n_uniform_convergence} it follows that
 \begin{equation}
 \lim_{n \rightarrow \infty}g(r_n(r_0, \theta, u/n^{1/4})) = g(r(r_0, \theta, 0)) = g(r_0)
 \end{equation}
 and
 \begin{equation}
\lim_{n \rightarrow \infty} v(r_0, \theta,u/n^{1/4}) = v(r_0,\theta,0) :=v(r_0)
 \end{equation}
 both uniformly on $\mathcal{O}$.  As a result, $\Pr\left(n^{1/4}N^{(n)}(r_0, \theta)>t\right)$ converges uniformly to $\exp \left( -g(r_0)v(r_0)
 t\right)$
 on $\cO$, which implies
 \begin{equation}
  \lim_{n \rightarrow \infty} \sup_{(r_0, \theta) \in S \times [-\pi, \pi]} \left| \E \left[ n^{p/4}N^{(n)}(r_0,\theta)^p \wedge M^p
  \right] -
  \int^{M}_0pt^{p-1}\exp\left(-g(r_0)v(r_0)t\right)\right| = 0.
 \end{equation}
It is then standard to extend the result to the expectation without the truncation.

 \end{proof}
For the cases $p=1,2$ we have by Lemma \ref{l_N_moments}
 \begin{equation}\label{e_EN}
 \lim_{n\rightarrow \infty}\sup_{(r_0,\theta)\in S\times [-\pi,\pi]}\left|\E\left[n^{1/4}N^{(n)}(r_0, \theta)\right] -  \frac{1}{g(r_0) v(r_0)} \right|= 0
 \end{equation}
 and
 \begin{equation}\label{e_EN^2}
  \lim_{n\rightarrow \infty} \sup_{(r_0,\theta)\in S\times [-\pi,\pi]}\left|\E\left[n^{1/2}\left(N^{(n)}(r_0, \theta)\right)^2\right] - \frac{2}{g^2(r_0)
  v^2(r_0)}\right|=0.
 \end{equation}

The next result shows that on the event $\{N^{(n)}>\varepsilon\}$ the $k$th moment of the displacement of the particle at time $0$ and at the first
reflection $N^{(n)}$ decays faster than $\frac{1}{n^m}$ as $n\rightarrow \infty$.

\begin{lemma}\label{l_euclidean_jump_convergence}
Fix $\varepsilon >0$ ,  $m,k \in \mathbb{N}$. Then
\[
\lim_{n \rightarrow \infty } \sup_{\{ (\bx_0,\bu) : \|  \bx_0\|   \in S, \bu \in \mathbf{S}^2\}}n^m \E \left[\left|\left| \by_n\left(N^{(n)}(\bu)\right) -\bx_0\right| \right|^k\mathds{1}_{\{N^{(n)} > \varepsilon\}} \right] = 0.
\]
\end{lemma}

\begin{proof}
Since $\|  \by'\|   = v_n(\by)$ we have
\begin{multline*} \E \left[\left|\left| \by_n\left(N^{(n)}(\bu)\right) -\bx_0\right| \right|^k\mathds{1}_{\{N^{(n)} > \varepsilon\}} \right] \leq \E \left[\left(\int_0^{N^{(n)}} v_n(\by(s)) ds\right)^k\mathds{1}_{\{N^{(n)} > \varepsilon\}} \right] \\
 \leq n^{-k/2} g_{min}^{-k} \left(\E \left[\left(\int_0^{N^{(n)}} g_n(\by(s)) v_n(\by(s)) ds\right)^{2k}\right]\right)^{1/2}  \P(N^{(n)} > \varepsilon)^{1/2}.
\end{multline*}
The lemma now follows from Corollary \ref{c_PN_limit} and the fact that
\[\int_0^{N^{(n)}} g_n(\by(s)) v_n(\by(s)) ds\]
is distributed like an $\mathrm{Exp}(1)$ random variable and, consequently, has finite moments.
\end{proof}

\begin{lemma}\label{l_ENU}
Let $\bU$ be uniformly distributed on $\bS^2$.  Then
\begin{equation}
\lim_{n \rightarrow \infty} \sup_{\bx: \|  \bx\|   \in S}\left\| n^{3/4}   \E \left[ N^{(n)}(\bx,\bU)\bU
\right]  - \frac{1}{3g(\bx)^2v(\bx)}\left(\frac{ \nabla \cU(\bx)}{m v(\bx)^2} - \frac{\nabla g(\bx)}{g(\bx)}   \right) \right\|= 0.
\end{equation}
\end{lemma}

\begin{proof}
First, for notational convenience, define the auxiliary function
 \begin{equation}\label{e_F_n_definition}
\begin{split} F_n(\bx , \bu, t) &  := \int^t_0 g_n(\by_n(\bx, \bu, s))v_n(\by_n(\bx, \bu,s))ds \\
& = \int^t_0 g_n(r_n(\|  \bx\|  , \theta, s))v_n(r_n(\|  \bx\|  , \theta,s))ds,\end{split}
 \end{equation}
where $\theta = \cos^{-1}( \bu\cdot \bx/\|  \bx\|  )$.
In the usual way, we will often suppress the $\bx$ and $\bu$ dependencies of $F_n$ and write $F_n(t)$
when no confusion will arise.   As previously noted, $F_n(\bx,\bu,N^{(n)}(\bx,\bu))$ has an $\mathrm{Exp}(1)$ distribution, so
 \begin{equation}\label{e_1=EF}
 1 = \E[F_n(\bx,\bu,N^{(n)}(\bx,\bu)) ]
 \end{equation}
 and
 \begin{equation}\label{e_2=EF^2}
 2 = \E[F_n(\bx,\bu,N^{(n)}(\bx,\bu))^2].
 \end{equation}

In order to prove uniform convergence, we need to work on a compact set, so we fix a time $T>0$ and split the expectation and use the fact that $\E(\bU) = (0,0,0)$
\begin{equation}\label{e_EFU_split}
0 = \E[U_1F_n(\bx, \bU, N^{(n)}(\bU)) \mathds{1}_{\{N^{(n)}(\bU)< T\} }] + \E[U_1F_n(\bx, \bU, N^{(n)}(\bU))
\mathds{1}_{\{N^{(n)}(\bU) \geq T\}
}]
\end{equation}
By symmetry, it is enough to consider the first coordinate.  Taylor expanding $F_n(\bx, \bu, t)$ about $t=0$ yields
\begin{equation}
F_n(\bx, \bu, t) = g_n(\bx)v_n(r_0) t  + \frac{1}{2} \ddot F_n(\bx,\bu, \tau(t))t^2
\end{equation}
for some $\tau(t) \in [0,t]$.  Here $\ddot F_n$ denotes the second derivative with respect to time $t$.  Consequently
\begin{equation}
\begin{split}
&\E \left[ U_1 F_n(\bx, \bU, N^{(n)}(\bU))\mathds{1}_{\{N^{(n)}(\bU) < T\} }\right]\\
 &= g_n(\bx)v_n(\bx)\E \left[ U_1N^{(n)}(\bU)\mathds{1}_{\{N^{(n)}(\bU) < T\} }\right]+ \frac{1}{2}\E \left[ U_1 \ddot F_n(\tau)
 N^{(n)}(\bU)^2\mathds{1}_{\{N^{(n)}(\bU) < T\} } \right].
\end{split}
\end{equation}
If we substitute this into \eqref{e_EFU_split} we can write
\begin{equation}
\begin{split}
 \E \left[ U_1 N^{(n)}(\bU)\mathds{1}_{\{N^{(n)}(\bU)< T\} } \right]
 &= \frac{-1}{g_n(\bx)  v_n(\bx)} \Bigg[ \E\left[ U_1F_n(\bx, \bU, N^{(n)}(\bU)) \mathds{1}_{\{N^{(n)}(\bU) \geq T\} }\right]\\
 &\qquad+\frac{1}{2}\E \left[ U_1 \ddot F_n(\tau) (N^{(n)}(\bU))^2  \mathds{1}_{\{N^{(n)}(\bU)< T\} }\right] \Bigg],
 \end{split}
\end{equation}
and so
\begin{equation}\label{e_n3/4NU_expansion}
\begin{split}
 n^{3/4}\E \left[ U_1N^{(n)}(\bU)\right]&=
 \frac{-n^{1/2}}{g(\bx)v(\bx)} \Bigg[  \E\left[U_1F_n(\bx, \bU, N^{(n)}(\bU)) \mathds{1}_{\{N^{(n)}(\bU) \geq T\} }\right]  \\
 &\qquad +\frac{1}{2}\E\left[ U_1 \ddot F_n(\tau) (N^{(n)}(\bU))^2  \mathds{1}_{\{N^{(n)}(\bU)< T\} }\right] \Bigg] + n^{3/4} \E \left[ U_1 N^{(n)}(\bU)\mathds{1}_{\{N^{(n)}(\bU)\geq T\} } \right].
 \end{split}
\end{equation}
 Next we compute the limit as $n \rightarrow \infty$ of each term in the expansion \eqref{e_n3/4NU_expansion}.  The first and last term on the right hand side can easily be seen to go to $0$ as follows: By  applying Corollary \ref{c_EN_tails}, it follows that
  \begin{equation}\label{e_En3/4N_tails}
  \lim_{n \rightarrow \infty} \sup_{(\bx,\bu) : \|  \bx\|  \in S}n^{3/4}\E \left[N^{(n)}(\bx, \bu)\mathds{1}_{\{N^{(n)}(\bu)\geq T\} }\right] = 0.
  \end{equation}
  Similarly, by Cauchy-Schwarz, \eqref{e_2=EF^2} and by an application of Corollary \ref{c_PN_limit} with $k=1$ we have
  \begin{equation}\label{e_EFn_tails}
  \begin{split}
  n^{1/2}\E \left[ F_n(\bx, \bu, N^{(n)}(\bu)) \mathds{1}_{\{N^{(n)}(\bu) \geq T\} }\right]  &\leq n^{1/2}\sqrt{ \E [F^2_n(\bx,
  \bu,
  N^{(n)}(\bu) ]\Pr\{N^{(n)}(\bu) \geq T\}}\\
  & = \sqrt{2}\sqrt{ n\Pr\{N^{(n)}(\bu) \geq T\}} \\
  &\rightarrow 0
  \end{split}
  \end{equation}
 uniformly for $(\bx, \bu)$ such that $\|  \bx\|  \in S$ as $n \rightarrow \infty$.

We now deal with the middle term.  Differentiating equation \eqref{e_F_n_definition}  twice yields
\begin{equation*}
\begin{split}
\ddot F_n(t) = v_n(\by_n(t)) \nabla g_n(\by_n(t))\cdot {\dot \by_n(t)} + g_n(\by_n(t))\nabla v_n(\by_n(t))\cdot \dot \by_n(t)\\
= n^{1/4} \dot \by_n(t)\cdot \Big(v_n(\by_n(t)) \nabla g_n(\by_n(t)) + g_n(\by_n(t))\nabla v_n(\by_n(t)) \Big).
\end{split}
\end{equation*}
To evaluate $\lim_{n \rightarrow \infty} n^{1/4}\dot r_n(t)$, note that by \eqref{eq e evo} and  Lemma \ref{l_r_n_uniform_convergence}

\begin{equation}\label{e_dot_r_n_convergence}
\begin{split}
\|   n^{1/4}\dot\by_n(t) - v(\bx) \bu \|   &= n^{1/4}\|   \dot \by_n(t) - \dot \by_n(0)\| \\
&\leq n^{1/4}\int^t_0 \| \ddot \by_n(s)\|ds \\
&= \frac{1}{n^{1/4}} \int^t_0 \sqrt{n} \| \ddot \by_n(s)\| ds\\
& \rightarrow 0
\end{split}
\end{equation}
uniformly for $(\bx, \bu, t) \in \{\bx :\|\bx\|\in S\} \times \bS^2\times [0,T]$ as $n \rightarrow \infty$.  Differentiating \eqref{eq E speed} shows
\[
\nabla v(\bx) =- \frac{\nabla \cU(\bx)}{mv(\bx)},
\]
which together with \eqref{e_dot_r_n_convergence},  Lemma \ref{l_r_n_uniform_convergence}, and the continuity of $\nabla v$ and $\nabla g$ on $\{ \bx : \|  \bx\|  \in S\}$ forces
\begin{equation}\label{e_ddot_F_N_limit}
\lim_{ n \rightarrow \infty} \ddot F_n(t) = \ddot F_1(0) = v(\bx)\bu\cdot \left( v(\bx) \nabla g(\bx) + g(\bx) \nabla v(\bx) \right)=\bu\cdot \left( v(\bx)^2 \nabla g(\bx) - g(\bx) \frac{\nabla \cU(\bx)}{m} \right)
\end{equation}
uniformly for $(\bx, \bu, t) \in \{\bx : \|   \bx\|  \in S\} \times\bS^2\times [0,T]$.
In conjunction with Lemma \ref{l_N_moments} this yields
\begin{equation*}
\begin{split}
&\sup_{(\bx,\bu) : \|  \bx\|  \in S, \bu\in\bS^2} \sqrt{n} \cdot \E \left[\left|  \ddot F_n(\bx, \bu,N^{(n)}(\bu)) -\ddot F_1(\bx,\bu,0)\right|
\left(N^{(n)}(\bu)\right)^2
\mathds{1}_{\{N^{(n)}(\bu)<T\} } \right]\leq \\
&~\sup_{(\bx,\bu) : \|  \bx\|  \in S, \bu\in\bS^2}\left\{ \left|  \ddot F_n(\bx,\bu, t) -\ddot F_1(\bx,0)\right|\right\} \sqrt{n}
\sup_{(\bx,\bu) : \|  \bx\|  \in S, \bu\in\bS^2}\E \left[
\left(N^{(n)}(\bu)\right)^2 \mathds{1}_{\{N^{(n)}(\bu)<T\} } \right] \\
&\rightarrow 0
\end{split}
\end{equation*}
as $n\rightarrow \infty$.
This combined with \eqref{e_EN^2} shows
\begin{equation}\label{e_E_ddot_F_n_limit}
\lim_{n \rightarrow \infty} \sup_{(\bx,\bu) : \|  \bx\|  \in S, \bu\in\bS^2}  \E \left[ \sqrt{n} \ddot F_n(\bx, \bu,N^{(n)}(\bu))
\left(N^{(n)}(\bu)\right)^2
\mathds{1}_{\{N^{(n)}(\bu)<T\} } -\frac{2}{g^2(\bx)v^2(\bx)} \ddot F_1(\bx,\bu,0)\right]=0.
\end{equation}
Combining these calculations with the expansion \eqref{e_n3/4NU_expansion}, we have
  \begin{equation}\label{e_nENcos}
 \begin{split}
\lim_{n\to\infty} n^{3/4}\E \left[ U_1N^{(n)}(\bU) \right]
 & =\lim_{n\to\infty} \frac{-n^{1/2}}{2g(\bx)v(\bx)}\E\left[ U_1 \ddot F_n(\tau) (N^{(n)}(\bU))^2  \mathds{1}_{\{N^{(n)}(\bU)< T\} }\right]  \\
 &= \frac{-1}{g(\bx)^3v(\bx)^3}\E\left[ U_1\bU\cdot \left( v(\bx)^2 \nabla g(\bx) - g(\bx) \frac{\nabla \cU(\bx)}{m} \right) \right] \\
& =  \frac{-1}{3g(\bx)^3v(\bx)^3}\left( v(\bx)^2 g_{x_1}(\bx) - g(\bx) \frac{ \cU_{x_1}(\bx)}{m} \right)
 \end{split}
 \end{equation}
This completes the proof.
\end{proof}


Let $(\bX^n_k,\cT^n_k)_{k\geq 0}$ be a Markov chain with transition operator
\[ Q_nf(\bx,t)  = \E\left[f\left(\by\left[\bx,v(\bx)\bU, N^n(\bx,\bU)\right], t+ n^{-3/4}N^n(\bx,\bU)\right) \right],\]
as defined in \eqref{eq transition time}.  Note that $(\bX^n_k)_{k\geq 0}$ is a Markov chain with transition operator \eqref{eq transition} (with appropriate scaling), while the second coordinate keeps track of the time between collisions.  That is, if we start the chain from $(\bx,0)$, then $n^{3/4}\cT^n_k = T^n_k$ is the time at which the $k$'th collision occurs.
\begin{lemma}\label{l_R_drift_convergence}
Let
\[
\mu_{n}(\bx,t) := n \E \left[ (\bX^n_1,\cT^n_1) - (\bx,t) \mid (\bX^n_0,\cT^n_0) = (\bx,t) \right]
\]
be the scaled drift of $(\bX^n_k,\cT^n_k)_{k\geq 0}$.  Then
\[
\lim_{n \rightarrow \infty} \sup_{(\bx,t): \|\bx\| \in S, t\geq 0} \left|\mu_{n}(\bx,t)  - \left(\frac{-1}{3g(\bx)^2}\left[\frac{ 2\nabla\cU(\bx)}{m v(\bx)^2} + \frac{\nabla g(\bx)}{g(\bx)}   \right] \ , \ \frac{1}{g(\bx)v(\bx)}\right)   \right| =0,
\]
uniformly on $\{\bx : \|\bx\| \in S\}\times [0,\infty)$.
\end{lemma}
\begin{proof}
The result for the time coordinate is simply a restatement of \eqref{e_EN}.  By symmetry, it suffices to prove convergence for the first spatial coordinate only.  By definition of $X^n_1$, this is
\[
n\E\left[X^n_1 - x_1 | (\bX^n_0,\cT^n_0) = (\bx,t)\right] = n \E \left[y_{n,1} \left( N^{(n)}(\bU) \right)  - x_1 \right].
\]
Let $\Lambda_\delta(S) :=\Lambda^1_\delta(S) $ as defined in  \eqref{e_lambda_delta_definition}.  Since $\Lambda^n_\delta(S)$ is increasing in $n$, we have by
construction
that $\|\by_n(t)\|=r_n(t) \in S^\prime$ for all $n \in \mathbb{N}$ and $t \in [0,\Lambda_\delta(S)]$.  So in particular, $\by_n(t)$ is bounded away from
$\partial \cD$.  To compute $\mu_{n,1}$ we first split the expectation on the events $\{N^{(n)}(\bU) \leq \Lambda_\delta(S)\}$ and
$\{N^{(n)}(\bU) > \Lambda_\delta(S)\}$. This allows us to write
\begin{equation}\label{e_mu_r,n}
\begin{split}
n \E \left[y_{n,1} \left( N^{(n)}(\bU) \right)  - x_1 \right] &= n \E \left[\left( y_{n,1} \left( N^{(n)}(\bU) \right)  - x_1 \right) \mathds{1}_{\{N^{(n)} \leq \Lambda_\delta(S)\}}\right]\\ &~+n\E
\left[\left(y_{n,1} \left( N^{(n)}(\bU) \right)  - x_1 \right) \mathds{1}_{\{N^{(n)} > \Lambda_\delta(S)\}}\right].
\end{split}
\end{equation}
By Lemma \ref{l_euclidean_jump_convergence}, the second term of \eqref{e_mu_r,n} converges to $0$ uniformly on $\{\bx :\|\bx\|\in S\}$.  To compute the first term of \eqref{e_mu_r,n}, we utilize a second order Taylor expansion of $y_{n,1}(t)$ evaluated at $t =N^{(n)}(\bU)$  which
yields
\begin{equation*}
\begin{split}
y_{n,1}\left(\bx,N^{(n)}(\bU)\right) -x_1&=v_n(\bx)U_1 N^{(n)}(\bU)  + \frac{1}{2}\ddot y_{n,1} (\tau) \left[N^{(n)}(\bU)  \right]
^2
\end{split}
\end{equation*}
for some $0 \leq \tau := \tau(\bx, N^{n}(\bU) ) \leq N^{n}(\bU)$.  Hence
\begin{equation}\label{e_mu_rn_expansion}
\begin{split}
n \E \left[\left( y_{n,1} \left( N^{(n)}(\bU) \right)  - x_1 \right) \mathds{1}_{\{N^{(n)} \leq \Lambda_\delta(S)\}}\right]&=
\E \left[n v_n(\bx)U_1 N^{(n)}(\bU) \mathds{1}_{\{N^{(n)} \leq \Lambda_\delta(S)\}}\right] \\
&~+\frac{n}{2}\E \left[ \ddot y_{n,1}(\tau)N^{(n)}(\bU)^2\mathds{1}_{\{N^{(n)} \leq \Lambda_\delta(S)\}}  \right].
\end{split}
\end{equation}
Note that
\[
\left|\E \left[n v_n(\bx)U_1\cdot N^{(n)}(\bU) \mathds{1}_{\{N^{(n)} > \Lambda_\delta(S)\}}\right]\right| \leq \sup_S |v(\bx)| n^{3/4} \E
\left[N^{(n)}(\bU) \mathds{1}_{\{N^{(n)} > \Lambda_\delta(S)\}} \right] \rightarrow 0
\]
uniformly for $\bx$ such that $\|\bx\| \in S$ as $n \rightarrow \infty$ by Corollary \ref{c_EN_tails}. Thus the limit of the first term of \eqref{e_mu_rn_expansion}  can be computed by
a direct application of Lemma \ref{l_ENU}.
\begin{equation}\label{e_mu_rn_first_term_limit}
\begin{split}
&\lim_{n \rightarrow \infty} \E \left[n v_n(\bx)U_1 \cdot N^{(n)}(\bU) \mathds{1}_{\{N^{(n)} \leq \Lambda_\delta(S)\}}\right] \\
&= \lim_{n \rightarrow \infty} v(\bx)\E \left[n^{3/4} U_1 \cdot N^{(n)}(\bU)\right] = \frac{1}{3g(\bx)^2}\left(\frac{ \cU_{x_1}(\bx)}{m v(\bx)^2} - \frac{g_{x_1}(\bx)}{g(\bx)}   \right)\end{split}
\end{equation}
uniformly on $\{\bx : \|\bx\|\in S\}$.
We now compute the limit of the second term on the right hand side of \eqref{e_mu_rn_expansion}.

Since $\tau(N^{(n)}(\theta)) \leq N^{(n)}(\theta)$,  by Lemma \ref{l_r_n_uniform_convergence} and Lemma \ref{l_bounded_moments}
\begin{equation*}
\begin{split}
&\sup_{(\bx,\bu) :\|\bx\|\in S,\bu\in\bS^2}\left|\E \left[ \left(\sqrt{n} \ddot y_{n,1}(\tau(N^{(n)}(\bu)))   + \frac{\cU_{x_1}(\bx)}{m} \right)\sqrt{n}
(N^{(n)}(\bu))^2\mathds{1}_{\{
N^{(n)} \leq\Lambda_\delta(S) \}} \right]\right| \leq \\
&\sup_{(\bx,\bu) :\|\bx\|\in S,\bu\in\bS^2} \left\{ \sup_{t \in [0, \Lambda_\delta(S)]} \Big| \sqrt{n}\ddot y_{n,1}( \bx, \bu,t) +
  \frac{\cU_{x_1}(\bx)}{m} \Big| \cdot \E\left[\sqrt{n}N^{(n)}(\bu)^2 \mathds{1}_{\{ N^{(n)} \leq \Lambda_\delta(S) \}} \right] \right\}\rightarrow 0
\end{split}
\end{equation*}
 as $n \rightarrow \infty$.
This together  with Corollary \ref{c_EN_tails} and Lemma \ref{l_N_moments} yields
\[
\lim_{n \rightarrow \infty} \sup_{(\bx,\bu) :\|\bx\|\in S,\bu\in\bS^2}\left|\E\left[n \ddot y_{n,1}(\tau(N^{(n)}(\bu))) N^{(n)}(\bu)^2\mathds{1}_{\{ N^{(n)}
\leq \Lambda_\delta(S) \}} \right] + \frac{2\cU_{x_1}(\bx)}{mg(\bx)^2v(\bx)^2}\right| = 0.
\]
Since this convergence is uniform in $\bu$, we can evaluate the limit of the second order term of equation \eqref{e_mu_rn_expansion} by
\begin{equation}\label{e_mu_rn_second_term_limit}
\lim_{n \rightarrow \infty} \E \left[ \frac{n}{2} \ddot y_{n,1}(\tau(N^{(n)}(\bU)))N^{(n)}(\bU)^2 \mathds{1}_{\{ N^{(n)}(\bU) \leq
\Lambda_\delta(S) \}}\right] =  -\frac{\cU_{x_1}(\bx)}{mg(\bx)^2v(\bx)^2}
\end{equation}
uniformly on $\{\bx:\|\bx\| \in S\}$.
  So by adding the
right hand sides of \ref{e_mu_rn_first_term_limit} and \ref{e_mu_rn_second_term_limit}
we have
\[
\lim_{n \rightarrow \infty} n\E\left[X^n_1 - x_1 | (\bX^n_0,\cT^n_0)=(\bx,t)\right]  =  -  \frac{1}{3g(\bx)^2}\left(\frac{ 2\cU_{x_1}(\bx)}{m v(\bx)^2} + \frac{g_{x_1}(\bx)}{g(\bx)}   \right)
\]
uniformly on $\{\bx:\|\bx\| \in S\}\times [0,\infty)$.

\end{proof}

\begin{lemma}\label{l_R_variance_convergence}
Let $((\bX^n_k,\cT^n_k))_{k\geq 0}$ be a Markov chain with transition operator as in \eqref{eq transition time} and let
\[\begin{split}
\sigma^2_{n,ij}(\bx,t) &:= n \E \left[ \left( X^n_{1,i} - x_i\right) \left( X^n_{1,j} - x_j\right) \mid (\bX^n_0,\cT^n_0) = (\bx,t) \right] \\[10pt]
\sigma^2_{n,it}(\bx,t) &:= n \E \left[ \left( X^n_{1,i} - x_i\right) \left( \cT^n_1- t\right) \mid (\bX^n_0,\cT^n_0) = (\bx,t) \right] \\[10pt]
\sigma^2_{n,t}(\bx,t) &:= n \E \left[ \left( \cT^n_1-t\right)^2  \mid (\bX^n_0,\cT^n_0) = (\bx,t) \right]
\end{split}\]
Then
\begin{enumerate}
\item\label{space}
$\displaystyle
\lim_{n \rightarrow \infty} \sup_{(\bx,t) : \|\bx\|\in S, t\geq 0 } \left| \sigma^2_{n,ij}(\bx,t) - \frac{2}{3g(\bx)^2}\delta_{ij} \right| =0,
$ where $\delta_{ij} = \mathds{1}\{i=j\}$.
\item \label{space/time}
$\displaystyle
\lim_{n \rightarrow \infty} \sup_{(\bx,t) : \|\bx\|\in S, t\geq 0 } \sigma^2_{n,it}(\bx,t)=0.
$
\item \label{time}
$\displaystyle
\lim_{n \rightarrow \infty} \sup_{(\bx,t) : \|\bx\|\in S, t\geq 0 } \sigma^2_{n,t}(\bx,t)=0.
$
\end{enumerate}
\end{lemma}

\begin{proof}
Note that \eqref{time} is immediate consequences of Lemma \ref{l_N_moments}, while \eqref{space/time} follows from \eqref{space}, \eqref{time} and an application of Cauchy-Schwarz,  It remains to establish \eqref{space}.
As in the proof of Lemma \ref{l_R_drift_convergence} we have
\begin{equation}\label{e_sigma_r,n}
\begin{split}
\sigma^2_{n,ij}(\bx,t) &= n \E \left[ \left( y_{n,i}(N^{(n)}(\bU)) - x_i\right)\left( y_{n,j}(N^{(n)}(\bU)) - x_j\right)\mathds{1}_{\{ N^{(n)} \leq \Lambda_\delta(S) \}} \right] \\
&~+ n \E\left[\left( y_{n,i}(N^{(n)}(\bU)) - x_i\right)\left( y_{n,j}(N^{(n)}(\bU)) - x_j\right)\mathds{1}_{\{ N^{(n)} > \Lambda_\delta(S) \}} \right]
\end{split}
\end{equation}
The second term on the right hand side goes to $0$ by Lemma \ref{l_euclidean_jump_convergence}.  To evaluate the limit of the first term on the right hand side of \eqref{e_sigma_r,n}, it is enough to use a first order Taylor expansion of $\by_n(t)$
\[
y_{n,i}(\bx, N^{(n)}(\bU)) - x_i = \dot y_{n,i}(\tau) N^{(n)}(\bU)
\]
for some $0 \leq \tau := \tau(\bx, N^{n}(\bU) ) \leq N^{n}(\bU)$.   Hence
\begin{multline}
n \E \left[ \left( y_{n,i}(N^{(n)}(\bU)) - x_i\right)\left( y_{n,j}(N^{(n)}(\bU)) - x_j\right)\mathds{1}_{\{ N^{(n)} \leq \Lambda_\delta(S) \}} \right] \\ =
 n \E \left[  \dot y_{n,i}(\tau)\dot y_{n,j}(\tau) N^{(n)}(\bU)^2\mathds{1}_{\{ N^{(n)} \leq \Lambda_\delta(S) \}} \right].
\end{multline}
Note that for $t\leq \Lambda_\delta(S)$, and $\bx$ such that $\|\bx\|\in S$, $||\by_n(\bx,\bu,t)|| \in S'$.  By the mean value theorem  and the fact that $\nabla\cU$ is bounded on $\{\mathbf{z} : \|\mathbf{z}\| \in S'\}$, there exists $C>0$ such that
\[ \sup_{(\bx,\bu) : \|\bx\|\in S, \bu\in \bS^2}\sup_{t\in[0, \Lambda_\delta(S)]} |n^{1/4}\dot y_{n,i}(\bx,\bu,\tau) - v(\bx)u_i| \leq \frac{2}{m n^{1/4}}\sup_{\mathbf{z} : \|\mathbf{z}\| \in S'} \cU_{x_i}(\mathbf{z})\Lambda_\delta(S) \leq \frac{C}{n^{1/4}} .  \]

Since $\tau(N^{(n)}(\theta)) \leq N^{(n)}(\theta)$,  we  can apply Lemma \ref{l_bounded_moments} and Lemma \ref{l_N_moments} to see that

\[ \begin{split} \lim_{n\to\infty}  n \E \left[  \dot y_{n,i}(\tau)\dot y_{n,j}(\tau) N^{(n)}(\bU)^2\mathds{1}_{\{ N^{(n)} \leq \Lambda_\delta(S) \}} \right]  & = \lim_{n\to \infty} v(\bx)^2 \E\left[ U_i U_j (n^{1/2}N^{(n)}(\bU)^2)\right]  \\
& = \frac{2}{g(\bx)^2} \E(U_iU_j)\\
& = \frac{2}{3 g(\bx)^2} \delta_{ij}
,\end{split}\]
uniformly on $\{\bx : \|\bx\| \in S\}$, which completes the proof.
\end{proof}

\section{Convergence of the free path process}\label{s_skeleton_process}
In this section we will prove the convergence of the free path process $\left(\bX^n_k,\cT^n_k\right)_{k \in \N_0}$ with transition operator $Q_n$ defined in \eqref{eq transition time}.  Recall that $\cD_{3} = \{\bx\in \R^3 : \|\bx\|\in \cD\}$.
The following is the main theorem of this section.
\begin{theorem}\label{t_step_process_convergence}
Let $(\cX_t)_{t\geq 0}$ be a diffusion on $\cD_{3}$ whose generator $G$ acts on functions $f \in C^2(\cD_{3})$ with compact support in $\cD_{3}^\circ$ by
\begin{equation}\label{e_G_definition}
Gf(\bx) =
\frac{1}{3g^2(\bx)}\Delta f(\bx)  - \frac{1}{3g^2(\bx)}\left(
\frac{\nabla g(\bx)}{g(\bx)}
+ \frac{2\nabla\cU(\bx )}{mv(\bx)^2}\right)\cdot \nabla f(\bx)
\end{equation}
and killed if/when the diffusion $\cR: = \|\cX\|$ hits the boundary of $\cD$.  Consider any $l,u \in \cD^\circ$ with
$l <  u $ and start the process   $\left(\bX^n_k,\cT^n_k\right)_{k \in \N_0}$ at $(\bx, 0)$, where $l<\|\bx\| < u$.  Define the
stopping times
\[\tau^n_{l,u} := \inf \{k \in \N_0: \|\bX^n_k\|  \notin [l,u] \}\] and
\[\tau_{l,u} := \inf \{ t \geq 0: \|\cX_t\|\notin [l,u]\}.
\]
Then, as $n\rightarrow \infty$, the family of continuous time processes $\left(\bX^{n}_{\floor{nt} \wedge \tau^n_{l,u}},\cT^{n}_{\floor{nt} \wedge \tau^n_{l,u}}\right)_{t
\geq 0}$
converges
in distribution on the Skorokhod space to the diffusion
\[  \left(\cX_{t \wedge \tau_{l,u}}, \int_0^{t \wedge \tau_{l,u}} \frac{ds}{g(\cX_s)v(\cX_s)}\right)_{t \geq 0 }\]
with initial conditions $(\bx,0)$.
\end{theorem}

As an immediate corollary we get the following result for the diffusion  $\cR$.
\begin{corollary}\label{c_radial}
Suppose $\cX$ is the diffusion defined in Theorem \ref{t_step_process_convergence}. Then  $\cR: = \|\cX\|$ is a diffusion on $\cD^\circ$ with a generator $G_r$ that acts on functions $f\in C^2(\cD)$ with compact support in $\cD^\circ$ by
\begin{equation}\label{e_G_r_definition}
G_r f(r) = \frac{1}{3g^2(r)}\partial^2_{rr}f(r) - \frac{1}{3g^2(r)}\left(\frac{\partial_r g(r)}{g(r)}-\frac{1}{r}+\frac{2\partial_r U(r)}{mv^2(r)}\right)\partial_r f(r).
\end{equation}
\end{corollary}
\begin{proof}
This follows by using Ito's Lemma to the diffusion $\cX$ from \eqref{e_G_definition} together with the following two observations:
\begin{itemize}
\item[a.] If $g(\bx)$ only depends on $\|x\|=r$ then
\[
\nabla g(\bx) = \partial_rg(r) \frac{\bx}{r}.
\]
\item[b.] If $\cX=:(X_1, X_2, X_3)$ and $(B_1, B_2, B_3)$ is a standard three dimensional Brownian motion then
\[
\sum_{i=1}^3 X_i dB_i/ \cR
\]
is a standard one dimensional Brownian motion. This follows from the fact that
\[
d\left[\sum_{i=1}^3 X_i dB_i/ \cR \right]_t = \sum_{i=1}^3\frac{ X_i^2}{\cR^2} dt = dt.
\]
\end{itemize}

\end{proof}

We prove Theorem \ref{t_step_process_convergence} at the end of this section.
The proof utilizes Theorem IX.4.21 from \cite{JS03} which gives sufficient conditions to prove a continuous time step process converges
to a diffusion.  We add this result below in order to make it easier for the reader to follow our proof.

\begin{theorem}\label{t_IX.4.21}

Suppose that for each $n \in \mathbb{N}$, $X^{n}$ is a pure jump Markov process.  That is, its generator has the form
\[
A^nf(x) = \int [f(x+y) - f(x)]K^n(x,dy)
\]
where $K^n$ is a finite transition kernel on $\mathbb{R}^d$.
Then define $b^n$ and $c^n$  by
\begin{equation}
b^n(x) = \int y K^n(x,dy),~c^{n,ij}(x)= \int y^iy^jK^n(x,dy).
\end{equation}
Let $b$, $c$ be continuous functions on $\R^d$ and suppose $X$ is a diffusion whose generator is given by
\begin{equation}\label{e_JS03_gen_def}
G f(x) = \sum_{i \leq d} b_i(x) D_i f(x)  + \frac{1}{2} \sum_{1 \leq i,j \leq d} c^{ij}(x)D_{ij} f(x)
\end{equation}
and defines a martingale problem with a unique solution (see Assumption IX.4.3 from \cite{JS03}). Assume that
\begin{enumerate}
\item[(i)] $b^n \rightarrow b$, $c^n \rightarrow c$ locally uniformly;
\item[(ii)] $\sup_{x: |x| \leq a} \int K^n(x,dy) |y|^2 \mathds{1}_{\{ |y| > \varepsilon\}} \rightarrow 0$ as $n \uparrow \infty$  for all
    $\varepsilon > 0$;
\item[(iii)] $\nu_n \rightarrow \nu$ weakly, where $\nu_n$ and $\nu$ are the initial distributions of $X^n_0$ and $X_0$ respectively.
\end{enumerate}
Then the laws $\mathcal{L}(X^n)$  converge weakly to $P = \int P_x\nu(dx)$, the law of the diffusion process $X$ started with the initial distribution $\nu$.
\end{theorem}

Because we need to stop the diffusion before it enters a domain where the particle's velocity is not bounded away from $0$, the easiest way to apply Theorem \ref{t_IX.4.21} to our situation is to introduce a cutoff version of the chains $((\bX^n_k,\cT^n_k))_{k\geq 0}$, that behaves like this chain on $\{(\bx,t) : l\leq \|\bx\|\leq u\}$ and obviously satisfies the conditions of Theorem \ref{t_IX.4.21} outside this domain.  To this end, fix $l^*,u^*\in \cD^\circ$ such that $l^*<l<u<u^*$, and let $\phi:\R\to[0,1]$ be a smooth function such that $\phi$ is identically $0$ in a neighborhood of $0$ and in a neighborhood of $\infty$, and $\phi(r)=1$ if $r\in [l^*,u^*]$ and let $\zeta(\bx) = (\|\bx\|\vee l)\wedge u$.  Moreover, define
\[ b^{lu}(\bx,t)= b^{lu}(\bx) =   \left(\frac{-1}{3g(\zeta(\bx))^2}\left[\frac{ 2\nabla(\phi \cU)(\bx)}{m v(\zeta(\bx))^2} + \frac{\nabla(\phi g)(\bx)}{g(\zeta(\bx))}   \right] \ , \ \frac{1}{g(\zeta(\bx))v(\zeta(\bx))}\right)\]
and
\[ c^{lu}_{ij}(\bx,t)=c^{lu}_{ij}(\bx) = \frac{2}{3g(\zeta(\bx))^2}\delta_{ij}\mathds{1}(i\leq 3).\]
Let $\mathbf{A}$ be uniformly distributed on $\{-1,1\}^3$, and define the transition operator $R^{lu}_n$ by
\[ R^{lu}_n f(\bx,t) = \E\left[ f\left((\bx,t) + n^{-1/2}\sqrt{c^{lu}_{11}(\bx)}(\mathbf{A},0) + n^{-1} b^{lu}(\bx) \right) \right].\]
The next two results follow easily from these definitions.

\begin{proposition}\label{prop well posed}
Let $\cG_{lu}$ act on $f\in C^2(\R^4)$ with compact support by
\begin{equation}\label{eq cutoff generator} \cG_{lu} f(\bx,t) = \sum_{i=1}^3 \frac{c^{lu}_{ii}(\bx)}{2} D_{ii}f(\bx,t) + \sum_{i=1}^4 b^{lu}_{i}(\bx) D_if(\bx,t).\end{equation}
The martingale problem for $\cG_{lu}$ is well posed.
\end{proposition}

\begin{proof}
The generator $\cG_{lu}$ is slightly degenerate because $c^{lu}_{44} \equiv 0$, but this is not a serious complication, as is easily seen from the SDE perspective.  In particular, note that if $(Z^1_t,Z^2_t,Z^3_t,Z^4_t)_{t\geq 0}$ solves the corresponding SDE
\begin{equation}\label{eq sde3}  dZ^i = \sqrt{c^{lu}_{ii}(Z^1,Z^2,Z^3)} dB^i + b^{lu}(Z^1,Z^2,Z^3)_i dt\end{equation}
for $1\leq i\leq 3$ and
\[ dZ^4 =  b^{lu}(Z^1,Z^2,Z^3)_4 dt, \]
then $Z^4$ is a measurable, deterministic function of $(Z^1,Z^2,Z^3)$ and $Z^4$ does not appear in the SDE for $(Z^1,Z^2,Z^3)$.  Thus we need only establish existence and uniqueness in law for the SDE \eqref{eq sde3} for $(Z^1,Z^2,Z^3)$.  However, this is standard since the drift coefficients are bounded and continuous and the diffusion matrix is continuous and positive definite, see e.g. \cite[Theorem 8.1.7]{EK86}.
\end{proof}

\begin{proposition}\label{prop cut1}
Let $(\bZ^n_k)_{k\geq 0}$ be a Markov chain on $\R^4$ with transition operator $R^{lu}_n$.  Then, uniformly on $\R^4$, we have
\begin{enumerate}
\item $\displaystyle \lim_{n\to\infty} n\E(\bZ^n_1 - \bz \mid \bZ^n_0 =\bz) = b^{lu}(\bz)$
\item $\displaystyle \lim_{n\to\infty} n\E((Z^n_{1,i}-z_i)(Z^n_{1,j}-z_j) \mid \bZ^n_0 =\bz) = c^{lu}_{ij}(\bz) $
\item $\displaystyle\lim_{n\to\infty} \max_{k\geq 0} \|\bZ^n_{k+1}-\bZ^n_k\| =0$ almost surely.
\end{enumerate}
\end{proposition}

To handle the cutoff, we will need the following elementary result about convergence of hitting times with respect to the Skorokhod topology.

\begin{lemma}\label{l_first_passage_times}

For $\gamma \in D[0,\infty)$ define $\tau_x(\gamma) = \inf\{t: \gamma(t)\geq x\}$ to be the first passage time of $\gamma$ across $x$.
Consider an $f\in C[0,\infty)$ and an $a>0$ such that
\[
0<\tau_a(f) < \infty\quad \text{and} \quad \inf\{ t> \tau_a(f) : f(t)>a\} = \tau_a(f).
\]
If $(f_n)_{n\geq 1}$ is a sequence in $D[0, \infty)$ such that $f_n\to f$ in $D[0,\infty)$, then $\tau_a(f_n) \rightarrow \tau_a(f)$.
\end{lemma}

\begin{proof}[Proof of Theorem \ref{t_step_process_convergence}]
Let us introduce the cutoff transition operator $Q^{lu}_n$ defined by
\begin{equation}\label{e_cut_off_transition_operator}
Q^{lu}_n f(\bx, t) = \begin{cases} Q_n(\bx,t) & \textrm{if } \|\bx\| \in [l,u] \\[5pt] R^{lu}_n(\bx,t) &  \textrm{if } \|\bx\| \notin [l,u].\end{cases} \end{equation}
Let $\left(\bX^{lu,n}_k,\cT^{lu,n}_k\right)_{k\geq 0}$ be a Markov chain with transition operator $Q^{lu}_n f(\bx, t)$ started from $(\bx,0)$ with $\|\bx\|\in (l,u)$ and let $\eta_n = \inf\{k : \| \bX^{lu,n}_k\| \notin [l,u]\}$.  By the definition of the cutoff transition operator, we have the identity in distribution
\begin{equation}\label{eq eq dist} \left(\bX^{lu,n}_{k\wedge \eta_n},\cT^{lu,n}_{k\wedge \eta_n}\right)_{k\geq 0} =_d \left(\bX^{n}_{k\wedge \tau^n_{l,u}},\cT^{n}_{k\wedge\tau^n_{l,u}}\right)_{k\geq 0},\end{equation}
so it suffices to prove the result for the cutoff chain $\left(\bX^{lu,n}_k,\cT^{lu,n}_k\right)_{k\geq 0}$.  In order to apply Theorem \ref{t_IX.4.21} directly, we Poissonize the chain.  That is, let $(\Gamma_t)_{t\geq 0}$ be a rate one Poisson process and consider
\[ \left(\bX^{lu,n}(t),\cT^{lu,n}(t)\right)_{t\geq 0} := \left(\bX^{lu,n}_{\Gamma_{t}},\cT^{lu,n}_{\Gamma_{t}}\right)_{t\geq 0}.\]
Since $\Gamma_{nt}/n \to t$ uniformly on compact sets almost surely, it is enough to prove the result for $\left(\bX^{lu,n}(t),\cT^{lu,n}(t)\right)_{t\geq 0}$, where the relevant stopping time is now
\[\eta^c_n = \inf\{ t : \| \bX^{lu,n}(t)\| \notin [l,u]\}.\]
Note that $\left(\bX^{lu,n}(nt),\cT^{lu,n}(nt)\right)_{t\geq 0}$ is a pure jump Markov process with generator
\[ \cG^{lu}_n f(\bx,t) = n(Q^{lu}_n - I)f(\bx,t),\]
see e.g. \cite[Proposition 19.2]{K02}.  For this process, condition (i) of Theorem \ref{t_IX.4.21} is satisfied with $b=b^{lu}$ and $c=c^{lu}$ by a direct combination of Proposition \ref{prop cut1}, Lemma \ref{l_R_drift_convergence}, and Lemma \ref{l_R_variance_convergence}.  That condition (ii) of Theorem \ref{t_IX.4.21} is satisfied is an immediate consequence of Proposition \ref{prop cut1}, Lemma \ref{l_R_variance_convergence} (for the time component), and Lemma \ref{l_euclidean_jump_convergence}, where to apply the last Lemma we note that, because velocity is bounded away from $0$ and $\infty$ on $[l,u]$, the condition that $N^{(n)} > \epsilon$ in Lemma \ref{l_euclidean_jump_convergence} is comparable to the condition that the size of the jump be at least $\epsilon$ in condition (ii) of Theorem \ref{t_IX.4.21}.  Consequently we have
\[\left(\bX^{lu,n}(nt),\cT^{lu,n}(nt)\right)_{t\geq 0} \Rightarrow  \left(\cX^{lu}_{t},\cT^{lu}_t\right)_{t\geq 0}\]
where $\left(\cX^{lu}_{t},\cT^{lu}_t\right)_{t\geq 0}$ is a diffusion with generator \eqref{eq cutoff generator} started from $(\bx,0)$.  Using the Skorokhod representation, we may assume this convergence happens almost surely.  It follows from Girsanov's theorem that $\|\cX^{lu}\|$ satisfies the hypotheses of Lemma \ref{l_first_passage_times} and, as a result, this convergence happens jointly with the convergence of $\eta^c_n/n$ to $\eta^{lu} = \inf\{ t : \|\cX^{lu}\| \notin [l,u]\}$.  Therefore
\[ \left(\bX^{lu,n}((nt) \wedge \eta_n^c),\cT^{lu,n}((nt) \wedge \eta_n^c)\right)_{t\geq 0} \Rightarrow  \left(\cX^{lu}_{t\wedge \eta^{lu}},\cT^{lu}_{t\wedge \eta^{lu}}\right)_{t\geq 0}\]
as $n\rightarrow \infty$.
From the SDE perspective, see Proposition \ref{prop well posed}, it is clear that the generator for $((\cX^{lu}_{t},\cT^{lu}_{t}))_{t\geq 0}$ agrees with the generator for
\[ \left(\cX_{t }, \int_0^{t} \frac{ds}{g(\cX_s)v(\cX_s)}\right)_{ t \geq 0 }\]
on smooth functions $f(\bx,t)$ with compact support in $\{\bx : \|\bx\|\in [l^*,u^*]\}\times \R$, and thus we have
\[ \left(\cX^{lu}_{t\wedge \eta^{lu}},\cT^{lu}_{t\wedge \eta^{lu}}\right)_{t\geq 0} =_d  \left(\cX_{t \wedge \tau_{l,u}}, \int_0^{t \wedge \tau_{l,u}} \frac{ds}{g(\cX_s)v(\cX_s)} \right)_{t \geq 0}.\]
As a result of de-Poissonizing the process and \eqref{eq eq dist}, we have
\[\left(\bX^{n}_{\floor{nt}\wedge \tau^n_{l,u}},\cT^{n}_{\floor{nt}\wedge\tau^n_{l,u}}\right)_{t\geq 0} \Rightarrow \left(\cX_{t \wedge \tau_{l,u}}, \int_0^{t \wedge \tau_{l,u}} \frac{ds}{g(\cX_s)v(\cX_s)}\right)_{t \geq 0}.\qedhere\]
\end{proof}

\section{The Process on its natural time scale}\label{s_full_trajectory}
 In this section, we study the convergence of the full trajectory of the particle.  To reconstruct the full path from the free path process, we also need to keep track of the direction of reflection. To this end, we need to keep track of the reflection
 times and the direction of reflection.  That is, we will look at the Markov process $\left(\bX^n_k,\cT^n_k, \bU^n_k \right)_{k \in \N_0 }$ started from $(\bx,0,0)$ with $\|\bx\|\in \cD$ and
 transition operator
 \[
 \hat{Q}_n f(\bx, t, \mu) =\E\left[f\left(\by\left[\bx,v(\bx)\bU, N^n(\bx,\bU)\right], t+ n^{-3/4}N^n(\bx,\bU), \bU\right) \right].
 \]
By  \eqref{e_full_trajectory_definition} the full trajectory $\left(\bX^n(t)\right)_{t \geq 0}$ of the particle with energy $E_n$ moving in potential $\cU_n$ and scattering density $g_n$ is then constructed by
 \begin{equation}\label{e_full_trajectory}
 \bX^n(t) := \by_n \left(\bX^n_k, v(\bX^n_k)\bU^n_k,t - n^{3/4}\cT^n_k \right)
 \end{equation}
 for $t \in [n^{3/4}\cT^n_k,n^{3/4}\cT^n_{k+1})$.

 The following theorem is the main result of this section.

 \begin{theorem}\label{t_full_trajectory_convergence}
 Let $(\bX^n(t))$ denote the full trajectory of the particle as defined in \eqref{e_full_trajectory} and let
 \[
 \iota^n_{l,u} :=\inf \{  t: \|\bX^n(t)\| \notin [l,u]\}.
 \]
For any fixed $l,u \in
 {\cD}^\circ$ with $l < u$, as $n\rightarrow \infty$ we have the following convergence
 in distribution on $D(\mathbb{R}_+, \mathbb{R}^3)$:
 \[
 \left( \bX^n((n^{3/4} t)\wedge \iota^n_{l,u}) \right)_{t\geq 0} \rightarrow \left(\cX(\Omega(t)\wedge
 \tau_{l,u})\right)_{t\geq 0}
 \]
 where $\cX$ is as in Theorem \ref{t_step_process_convergence} and $\Omega$ is the time change given by
 \[
 \Omega(t) := \cI \left(\int^{\cdot}_0\frac{ds}{\lambda( \cX(s \wedge  \tau_{l,u})) } \right)(t)
 \]
and $\cI$ is the inverse operator defined by $\cI(f)(t) = \inf \{ s: f(s) > t\}$.
 \end{theorem}

To prove a limit theorem for the full path, we must invert the time process $(\cT^n_k)_{k\geq 0}$.  This is easiest to do if we extend this process to $[0,\infty)$ by linear interpolation:
 \begin{equation}\label{e_T^n_s}
 \cT^n(s) := \cT^n_{\floor{s}} + (s - \floor{s})(\cT^n_{\floor{s}+1} - \cT^n_{\floor{s}}), ~\forall s \in [0, \infty).
 \end{equation}
Defining $\cI(f)(t) = \inf\{ s : f(s)>t\}$ we have
\begin{equation}\label{eq sampling}
 \bX^n(n^{3/4}\cT^n_k) = \bX^n_{\floor{\cI(\cT^n)(\cT^n_k)}} = \bX^n_k.
 \end{equation}
Consequently, we start by examining the convergence of $\left(\bX^n_{\floor{\cI(\cT^n)(t)}}\right)_{t\geq 0}$, and then prove that filling in the true values of the path on the intervals $[n^{-1}\cT^n_k,n^{-1}\cT^n_{k+1})$ does not affect the limit.  The cutoffs do affect the time change in a technical way.  In particular, the stopped process $(\cT^n(t \wedge \tau^n_{l,u}))_{t\geq 0} $ does not have invertible paths.  To accommodate this, we use the time change $\Omega_n: \mathbb{R}_+ \rightarrow \mathbb{R}_+$ defined by
\begin{equation}
\Omega_n(t) = \cI \left( \cT^n_{(n \cdot ) \wedge \tau^n_{l,u}} +   (\cdot - \frac{1}{n} \tau^n_{l,u})^+
\frac{1}{\lambda(\bX^n(\tau^n_{l,u}))}\right) (t),
\end{equation}
where $\lambda(\bx) = g(\bx)v(\bx)$.  The choice of the particular linear drift is because Theorem \ref{t_step_process_convergence}, together with the result in the proof about convergence of hitting times, implies that as $n\rightarrow \infty$
\begin{equation}\label{eq adjusted time} \left(\bX^n_{\floor{nt}\wedge \tau^n_{l,u}}, \cT^n_{(n t ) \wedge \tau^n_{l,u}} +   (t -n^{-1} \tau^n_{l,u})^+
\frac{1}{\lambda(\bX^n(\tau^n_{l,u}))}\right)_{t\geq 0}\Rightarrow \left(\cX_{t\wedge \tau_{l,u}},  \int_0^{t } \frac{ds}{\lambda(\cX_{s\wedge \tau_{l,u}})}\right)_{t\geq 0}.\end{equation}

The next result allows us to invert the time process.

\begin{lemma}\label{l_BR14_lemma_5.2}
If $f \in D( \mathbb{R}_+, \mathbb{R}_+)$ is continuous and strictly increasing with $\lim_{t \rightarrow \infty} f(t) = \infty$, then $\cI(f)\in D(\mathbb{R}_+, \mathbb{R}_+)$ and $\cI$ is continuous at $f$.
\end{lemma}

\begin{proof}[Proof of Theorem \ref{t_full_trajectory_convergence}]
By \eqref{eq adjusted time} and Lemma \ref{l_BR14_lemma_5.2} we have
\begin{equation}\label{eq inverted time} \left(\bX^n_{\floor{nt}\wedge \tau^n_{l,u}}, \Omega_n(t)\right)_{t\geq 0}\Rightarrow \left(\cX_{t\wedge \tau_{l,u}},  \Omega(t)\right)_{t\geq 0}.\end{equation}
Note that the process $(\Omega(t)_{t\geq 0}$ is supported on $C(\R_+,\R_+)$ and, consequently the composition map $(f,g)\mapsto f\circ g$ from $D(\R_+,\R^3)\times C(\R_+,\R_+) \to D(\R_+,\R^3)$ is continuous, see e.g. \cite[p. 151]{Bil99}.  This combined with \eqref{eq inverted time} implies that
\begin{equation}\label{eq composed time} \left(\bX^n_{\floor{\cI(\cT^n)(t)}\wedge \tau^n_{l,u}}\right)_{t\geq 0}=\left(\bX^n_{\floor{n \Omega_n(t)}\wedge \tau^n_{l,u}}\right)_{t\geq 0}\Rightarrow \left(\cX_{\Omega(t)\wedge \tau_{l,u}} \right)_{t\geq 0}\end{equation}
as $n\rightarrow \infty$.
For $T>0$ fixed, by \eqref{eq sampling}, we have that
\begin{multline*} \sup_{0\leq t\leq T} \left\|\bX^n\left(n^{3/4} \left(t\wedge \cT^n_{\tau^n_{l,u}}\right)\right)- \bX^n_{\floor{\cI(\cT^n)(t)}\wedge \tau^n_{l,u}}\right\| \\ \leq \sup_{k \leq \left\lceil n\Omega_n\left(T \wedge \cT^n_{\tau^n_{l,u}}\right)\right\rceil}\sup_{\cT^n_{k-1}\leq n^{-3/4}t \leq \cT^n_k}\left\|\by_n \left(\bX^n_{k-1}, \bU_{k-1},t - n^{3/4}\cT^n_{k-1}\right)-\bX^n_{k-1}\right\|.
\end{multline*}
Note that
\[n\Omega_n\left(T \wedge \cT^n_{\tau^n_{l,u}}\right) \leq n\Omega_n\left(\cT^n_{\tau^n_{l,u}}\right) = \tau^n_{l,u} =O(n)\]
in probability since $n^{-1}\tau^n_{l,u} \rightarrow \tau_{l,u}<\infty$ a.s.  Thus if we fix $\epsilon>0$, for $n$ sufficiently large we have
\begin{multline*} \P\left(\sup_{0\leq t\leq T} \left\|\bX^n\left(n^{3/4} \left(t\wedge \cT^n_{\tau^n_{l,u}}\right)\right)- \bX^n_{\floor{\cI(\cT^n)(t)}\wedge \tau^n_{l,u}}\right\| >\epsilon\right) \\
\leq \epsilon + \P\left(\sup_{k \leq n^2}\sup_{\cT^n_{k-1}\leq n^{-3/4}t \leq \cT^n_k}\left\|\by_n \left(\bX^n_{k-1}, \bU_{k-1},t - n^{3/4}\cT^n_{k-1}\right)-\bX^n_{k-1}\right\|\mathds{1}_{\{k\leq \tau^n_{l,u}+1\}} >\epsilon\right).
\end{multline*}
The probability on the right hand side can be made arbitrarily small by observing that by, Corollary \ref{c_PN_limit}, for $C>0$,
\[\P(n^{3/4}(\cT^n_k-\cT^n_{k-1}) > C | \|\bX^n_{k-1}\| \in [l,u]) \leq \sup_{(\bx,\bu): \|\bx\|\in [l,u], \bu\in \bS^2} \P(N^{(n)}(\bx,\bu)> C) = o(n^k)\]
in probability for every $k\geq 1$ and by Lemma \ref{l_r_n_uniform_convergence}, if $n$ is large enough then
\[\sup_{(\bx,\bu,t) \in \{\bx : \|\bx\|\in [l,u]\} \times \bS^2 \times [0,C]} \left\|\by_n(\bx,\bu,t) - \bx \right\| <\epsilon.\]
Therefore
\[ \left(\bX^n\left(n^{3/4} \left(t\wedge \cT^n_{\tau^n_{l,u}}\right)\right)\right)_{t\geq 0} \Rightarrow  \left(\cX_{\Omega(t)\wedge \tau_{l,u}} \right)_{t\geq 0},\]
as $n\rightarrow \infty$
and the fact that
\[ \left(\bX^n\left( (n^{3/4} t) \wedge\iota^n_{l,u}\right)\right)_{t\geq 0} \Rightarrow  \left(\cX_{\Omega(t)\wedge \tau_{l,u}} \right)_{t\geq 0},\]
as $n\rightarrow \infty$ follow from path continuity and observing that $\iota^n_{l,u} \leq n^{3/4}\cT^n_{\tau^n_{l,u}}$.
\end{proof}

\section{Classifying the boundaries of $\cD$}\label{s_boundaries}

In this section we give conditions under which the boundary points of $\cD$ are inaccessible, meaning the boundary points cannot be reached in finite time.
These results can then be used to remove the stopping at $u$ and/or $l$ in Theorem \ref{t_step_process_convergence_a}.

Suppose we have a regular diffusion $X$ with state space the interval $(\ell,r)$. Every diffusion has two basic characteristics: the speed measure $m(dx)$ and
the scale function $s(x)$.

We assume the infinitesimal generator $G:\cD(G)\mapsto\cC_b(I)$ of $X$ is a second order differential operator
\[
G f(x) = \frac{1}{2}\sigma^2(x)\partial_{xx}f(x) + \mu(x)\partial_x f(x)
\]
where $\sigma, \mu \in C(I)$ and $\sigma^2(x)>0$ for all $x\in I$. Let $B(x):=\int_a^x2\sigma^{-2}(y)\mu(y)\,dy$ for some arbitrary (fixed) $a\in I$.  Then it
is well-known that
\begin{itemize}
\item The speed measure is absolutely continuous with respect to Lebesgue measure and has density
\[
m'(x) = 2\sigma^{-2}(x)e^{B(x)}
\]
\item The scale function has density
\[
 s'(x)=e^{-B(x)}.
\]
\end{itemize}
 The domain $\cD(G)$ consists of all functions in $\cC_b(I)$ such that $G f\in\cC_b(I)$ together with the appropriate boundary conditions.

The boundary point $\ell$ is called \textit{accessible} when
\[
\int_\ell^x\left(\int_y^x m'(\eta)\,d\eta\right) s'(y)\,dy <\infty
\]
and \textit{inaccessible} when
\[
\int_\ell^x\left(\int_y^x m'(\eta)\,d\eta\right) s'(y)\,dy =\infty.
\]
Similarly, one can classify the boundary $r$.

As we have shown above in Corrolary \ref{c_radial}, the one-dimensional diffusion defining the limiting radial process has generator $\cG_r$ that acts on compactly supported functions in
$C^2(\cD^\circ)$ as
\begin{equation}
\begin{split}
\cG_r f(\rho) &:= \mu_r(\rho) f^{\prime}(\rho)+ \frac{\sigma^2_r(\rho)}{2}f^{\prime\prime}(\rho)\\
 &= \frac{1}{3g^2(\rho)}\left( -\frac{g^{\prime}(\rho)}{g(\rho)}
+\frac{1}{\rho} - \frac{\partial_r\cU(\rho)}{(E - \cU(\rho))}\right) f^{\prime}(\rho)+\frac{1}{3g^2(\rho)}f^{\prime \prime}(\rho).
\end{split}
\end{equation}
If $a\in \cD^\circ$ then the density of the scale function will be
\begin{equation*}
\begin{split}
s'(y)&=  \exp \left(- \int^y_a \frac{2 \mu_r(\rho)}{\sigma_r^2(\rho)} d \rho \right)\\
&= \exp \left(\int^y_a \left( \frac{g^{\prime}(\rho)}{g(\rho)}-\frac{1}{\rho} +\frac{\partial_r \cU(\rho)}{(E - \cU(\rho))} \right) d \rho \right)
\\
&= \exp \left( \ln(g(y)) - \ln(y) - \ln(E - \cU(y)) - C \right) \\
&=  \exp \left( \ln \left( \frac{g(y)}{y (E - \cU(y))}\right) - C
\right)\\
&= \frac{a (E - \cU(a))}{g(a)} \left(  \frac{g(y)}{y (E - \cU(y))}\right)\\
\end{split}
\end{equation*}
where $C := \ln \left( \frac{g(a)}{a (E - \cU(a))}\right)$.  As a result, if we fix an arbitrary $c\in \cD^\circ$ the scale function will be given by

\begin{equation}\label{e_scale}
\begin{split}
s(x) &= \int^x_c s'(y)dy= \frac{a (E - \cU(a))}{g(a)} \int^x_c \left(  \frac{g(y)}{y (E - \cU(y))}\right) dy.
\end{split}
\end{equation}
The speed measure density for the radial diffusion will be
\begin{equation}
m'(x) =\frac{2}{\sigma_r^2(x)s'(x)} =\frac{2g(a)}{a(E - \cU(a))} g(x)x (E - \cU(x)) .
\end{equation}
One can then find the speed measure by setting
\[
m(J) := \int_J m'(x) dx = \frac{2g(a)}{a (E - \cU(a))}\int_J g(x)x (E - \cU(x)) dx
\]
for any Lebesque measurable $J \subseteq D^\circ$.
\begin{remark}\label{r_scale}
If $s(\ell)=-\infty$ then $\ell$ cannot be reached in finite time and is therefore inaccessible. Similarly, if $s(r)=\infty$ then $r$ is inaccessible.
\end{remark}

\begin{proposition}\label{p_accessible_inaccessible}
Let $D^\circ=(h_-,h_+)$.
\begin{itemize}
\item[i)] Suppose that there exists $\alpha\geq 1$ such that when $y\downarrow h_-$
\[
\frac{1}{y(E-\cU(y))} = \mathcal{O}\left((y-h_-)^{-\alpha}\right).
\]
Then $h_-$ is inaccessible.
\item[ii)] Suppose that there exists $\alpha< 1$ such that when $y\downarrow h_-$
\[
\frac{1}{y(E-\cU(y))} = \mathcal{O}\left((y-h_-)^{-\alpha}\right).
\]
Then $h_-$ is accessible.
\end{itemize}

In particular, if $h_-=0$ and $|\cU(0)|<\infty$ then $0$ is inaccessible. Analogous results hold for $h_+$.
\end{proposition}
\begin{proof}
The main parts of the proposition follow easily from \eqref{e_scale} and Remark \ref{r_scale}.

If $h_-=0$ and $|\cU(0)|<\infty$ then, since by assumption $E>\cU(0)$, we have
\[
\frac{1}{y (E-\cU(y))} = \mathcal{O}\left(y^{-1}\right).
\]
\end{proof}

\subsection{Constant force} Suppose we have a constant force field
\[
\cU(x) = C x
\]
and $\cD = [0,h]$ for some $h>0$. We know that the speed has to be zero at $h$ so
\[
 E-\cU(h) = E-Ch= 0.
\]
This forces $E=Ch$ and thus
\[
\frac{1}{y(E-\cU(y))} = \frac{1}{yC(h-y)}.
\]
By Proposition \ref{p_accessible_inaccessible} both 0 and $h$ are inaccessible boundary points.
\subsection{Newtonian gravity} We consider the case when the potential is that of a gravitational force directed towards the origin.  That is, the potential
function is given by

\[
\cU(\rho) = -\frac{k}{\rho}
\]
for some constant $k >0$. Suppose the the total energy of the particle is positive $E>0$. We have
\[
s(x) =\frac{(Ea+k)a}{g(a)}\int^x_c \left(  \frac{g(y)}{(Ey+k)y}\right) dy.
\]
Since $g$ is bounded above and bounded away from $0$ on $\mathbb{R}_+$, for $y \downarrow 0$,
\[
\frac{g(y)}{\sqrt{(Ey+k)y}} = \mathcal{O}\left(\frac{1}{y}\right)
\]
which implies $\lim_{x \downarrow 0} s(x) = -\infty. $  For $y \uparrow \infty$ one has
\[
\frac{g(y)}{(Ey+k)y} = \mathcal{O}\left(\frac{1}{y^2}\right)
\]
so that $\lim_{x \uparrow \infty} s(x) = \infty$.
By Proposition \ref{p_accessible_inaccessible} the boundary points 0 and $\infty$ are inaccessible.

\section{Higher dimensions and general potentials}\label{s_cartesian}

Although our model was motivated by the three dimensional Lorentz gas, the process with uniform reflection direction makes sense in any dimension and our methods apply essentially without change.  In this section we explain how our results could be generalized to higher dimensions and to non spherically symmetric potentials. Consider a particle in $\R^d$ for $d\in\N$ under the influence of a force field with potential energy $\cU(\pmb {x})$. Let
$S^{d-1}$ be the unit sphere in $\R^d$. In this model upon reflection at $\pmb x = (x_1,\dots,x_d)$, the particle starts its path in the uniform direction
$\pmb u =(u_1,\dots,u_d)\in S^{d-1}$ and then travels under the influence of $\cU(\pmb x)$. The only difference is that some of the constants change
because we are integrating $\pmb u$ over $\bS^{d-1}$ instead of $\bS^2$. Using the fact that if $\pmb u = (u_1,\dots,u_d)$ is
uniformly distributed on $\bS^{d-1}$ then $\E u_i^2 = \frac{1}{d}$ for all $i\in\{1,\dots,d\}$ yields the following results.

Let $((\bX^n_k,\cT^n_k))_{k\geq 0}$ be the free path Markov chain, where $\bX^n_k$ is the location of the particle at the time of the $k$'th reflection and $n^{3/4}\cT^n_k$ is the time of the $k$'th collision.
\begin{theorem}\label{theorem chain}
Let
\[
\mu_{n}(\bx,t) := n \E \left[ (\bX^n_1,\cT^n_1) - (\bx,t) \mid (\bX^n_0,\cT^n_0) = (\bx,t) \right]
\]
be the scaled drift of $(\bX^n_k,\cT^n_k)_{k\geq 0}$, and let $S$ be a compact subset of $\cD^\circ$. Suppose the equivalents of Assumptions \textbf{(A1)}-\textbf{(A5)} are satisfied.  Then
\[
\lim_{n \rightarrow \infty} \sup_{(\bx,t): \|\bx\| \in S, t\geq 0} \left|\mu_{n}(\bx,t)  - \left(\frac{-1}{dg(\bx)^2}\left[\frac{ (d-1)\nabla\cU(\bx)}{m v(\bx)^2} + \frac{\nabla g(\bx)}{g(\bx)}   \right] \ , \ \frac{1}{g(\bx)v(\bx)}\right)   \right| =0,
\]
uniformly on $\{\bx : \|\bx\| \in S\}\times [0,\infty)$.  Furthermore, let
\[\begin{split}
\sigma^2_{n,ij}(\bx,t) &:= n \E \left[ \left( X^n_{1,i} - x_i\right) \left( X^n_{1,j} - x_j\right) \mid (\bX^n_0,\cT^n_0) = (\bx,t) \right] \\[10pt]
\sigma^2_{n,it}(\bx,t) &:= n \E \left[ \left( X^n_{1,i} - x_i\right) \left( \cT^n_1- t\right) \mid (\bX^n_0,\cT^n_0) = (\bx,t) \right] \\[10pt]
\sigma^2_{n,t}(\bx,t) &:= n \E \left[ \left( \cT^n_1-t\right)^2  \mid (\bX^n_0,\cT^n_0) = (\bx,t) \right].
\end{split}\]
Then
\begin{enumerate}
\item\label{space}
$\displaystyle
\lim_{n \rightarrow \infty} \sup_{(\bx,t) : \|\bx\|\in S, t\geq 0 } \left| \sigma^2_{n,ij}(\bx,t) - \frac{d-1}{dg(\bx)^2}\delta_{ij} \right| =0,
$ where $\delta_{ij} = \mathds{1}\{i=j\}$.
\item \label{space/time}
$\displaystyle
\lim_{n \rightarrow \infty} \sup_{(\bx,t) : \|\bx\|\in S, t\geq 0 } \sigma^2_{n,it}(\bx,t)=0.
$
\item \label{time}
$\displaystyle
\lim_{n \rightarrow \infty} \sup_{(\bx,t) : \|\bx\|\in S, t\geq 0 } \sigma^2_{n,t}(\bx,t)=0.
$
\end{enumerate}
\end{theorem}

\begin{theorem}\label{t_main_theorem_cartesian_d_dim}
Let $\left(\mathscr{X}(t),t\geq 0\right)$ be a diffusion on $\cD\subset \R^d$ whose generator $\mathscr{G}$ acts on functions $f \in C^2(\cD)$
with compact support in $C^2(\cD^\circ)$ by
\begin{equation*}
\begin{split}
\mathscr{G}f(\pmb x) &=
\frac{v(\pmb x)}{dg(\pmb x)}\Delta f(\pmb x) - \frac{v(\pmb x)}{g(\pmb x)}\left(
\frac{(d-1)\nabla \cU(\pmb x )}{dmv^2(\pmb x)}+\frac{\nabla g(\pmb x)}{dg(\pmb x)}
\right) \cdot \nabla f(\pmb x). \\
\end{split}
\end{equation*}
Let $\left(\bX^n(t), t \geq 0\right)$ be the trajectory of the particle and assume the equivalents of Assumptions \textbf{(A1)}-\textbf{(A5)} above hold.  Fix  a
set $ \mathcal{J} \subset \cD^\circ$ and
define
\[
\iota^n_{\mathcal{J}} = \inf \{t \geq 0: \bX^n(t)\notin \mathcal{J} \}
\]
and
\[
\tau_{\mathcal{J}} = \inf \{t \geq 0: \mathscr{X}(t)\notin
\mathcal{J}\}.
\]
Then as $n \rightarrow \infty$ we have the following convergence in distribution
\[
\left(\bX^n(n^{3/4}t \wedge\iota^n_{\mathcal{J}} ), t \geq 0\right) \rightarrow
\left(\mathscr{X}(t\wedge \tau_{\mathcal{J}}), t \geq 0\right).
\]
Futhermore, if the boundary of $\cD$ is inaccessible then we can remove the stopping in the above.
\end{theorem}

{\bf Acknowledgments.}  We thank Simon Griffiths,  Cristina Costantini and Steve Evans for very helpful discussions.

\bibliographystyle{amsalpha}
\bibliography{particle_e}

\end{document}